\documentclass{article}
\usepackage[pdftex]{graphicx}
\usepackage{amsmath,amssymb,amsthm}
\usepackage{bm}
\usepackage{ascmac}
\usepackage{mathtools}
\usepackage[all]{xy}
\usepackage{standard}
\usepackage{basicnewcommand}
\usepackage{url}

\newcommand{\Hi}{H_\text{i}}
\newcommand{\Ht}{H_\text{t}}
\newcommand{\Gi}{G_\text{i}}

\newcommand{\Gt}{G_\text{t}}
\newcommand{\fqt}{\fq_\text{t}}
\newcommand{\fht}{\fh_\text{t}}
\newcommand{\fgt}{\fg_\text{t}}

\newcommand{\fgi}{\fg_\text{i}}
\newcommand{\fqi}{\fq_\text{i}}
\newcommand{\fhi}{\fh_\text{i}}

\makeatletter
\@addtoreset{figure}{section}
\newcommand{\claim}{{\bf Claim. }}

\begin{document}

\title{Four-dimensional compact Clifford--Klein forms of pseudo-Riemannian symmetric spaces with signature (2, 2)}
\author{Keiichi Maeta}

\maketitle

\begin{abstract}
  We give a classification of irreducible four-dimensional symmetric spaces which admit compact Clifford--Klein forms. For this, we develop a method that applies to particular 1-connected solvable symmetric spaces.\par
 We also examine a `solvable analogue' of Kobayashi's conjecture for reductive groups and find an evidence that the reductive assumption in Kobayashi's conjecture is crucial.
\end{abstract}

\tableofcontents

\section{Introduction}

 We are interested in the classification of indecomposable pseudo-Riemannian symmetric spaces which admit compact Clifford--Klein forms. In this paper, we classify the spaces whose dimensions are up to four and transvection groups are solvable.
In the following, we review a background of our problem from two different viewpoints.\par
First, we review a classification of pseudo-Riemannian symmetric spaces. After \'E. Cartan (\cite{cartan1,cartan2}) classified Riemannian symmetric spaces, Berger (\cite{B}) gave a classification theory of irreducible symmetric spaces. However, since pseudo-Riemannian symmetric spaces have a degenerate subspace in their tangent spaces, they are not necessarily decomposed into irreducible symmetric spaces. A `minimum unit' of pseudo-Riemannian symmetric space is said to be indecomposable. Therefore, one may expect to classify indecomposable symmetric spaces. Indecomposable symmetric spaces with signature $(n,1)$ and $(n,2)$ are classified by Cahen, Wallach, Parker, Kath and Olbrich (\cite{CW,CP,OK}). 
 \par
Second, we review the existence problem of compact Clifford--Klein forms. For a Lie group $G$, its closed subgroup $H$ assume a discrete subgroup $\Gamma\subset G$ acts on $G/H$ properly discontinuously. We say the quotient space $\Gamma \bs G/H$ is a Clifford--Klein form (see Definition~\ref{Def:ckdefi}). In the late 1980s, a systematic study of Clifford--Klein forms for non-Riemannian homogeneous spaces was initiated by T. Kobayashi (\cite{kob89}). The following problem is one of the central problems in this field, but the final answer remains open.
\begin{prob}[\cite{kob89}]\label{Prob:cptckform}
 Classify homogeneous spaces $G/H$ which admit compact Clifford--Klein forms.
\end{prob}
Any Riemannian spaces of reductive type admit compact Clifford--Klein forms (\cite{borel}). However, this problem for non-Riemannian spaces is open. By the classification of Berger (\cite{B}), irreducible symmetric spaces are of reductive type, and most of the works on this problem have focused on classifying symmetric spaces of reductive type which admits compact Clifford--Klein forms (see \cite{Benoist,kob92a,kob96,kobayashiono,KY,Kulkarni,Labourie} and so on). Five series and seven sporadic irreducible symmetric spaces have been found to admit compact Clifford--Klein forms so far (\cite{KY}).

In view of the above two prospects, we are interested in the following subproblem of Problem \ref{Prob:cptckform}.
\begin{prob}[{\cite[\S 1]{OKCK}}]\label{mainprob} 
 Classify reducible and indecomposable pseudo-Riemannian symmetric spaces $G/H$ which admit compact Clifford--Klein forms.
\end{prob}

Since different groups act on a symmetric space transitively, symmetric spaces can be written as different forms. Therefore we can consider Problem~\ref{mainprob} for each group. 
For example, the following two groups act on a symmetric space transitively.

\begin{defi}[isometry group, transvection group, \cite{CP}]
 The {\it isometry group} $\Gi$ of a pseudo-Riemannian symmetric space is the group which consists of all transformations which preserve the pseudo-Riemannian metric on the space.
 The {\it transvection group} $\Gt$ of a pseudo-Riemmanian symmetric space is the closed and connected subgroup of the isometry group generated by the products of two geodesic symmetries.\par
\end{defi}

The transvection group $\Gt$ is a normal subgroup of the isometry group $\Gi$. Roughly speaking, the transvection group is the ``smallest'' group which acts on the symmetric space transitively. 
For example,
\begin{itemize}
 \item for Euclid space $\R^n$, we have $\Gi=O(n)\ltimes \R^n$ and $\Gt=\R^n$,
 \item for the sphere $S^n$, we have $\Gi=O(n+1)$ and $\Gt=SO(n+1)$.
\end{itemize}

Problem \ref{mainprob} was studied by Kath--Olbrich (\cite{OKCK}). They found a necessary and sufficient condition for the existence of compact Clifford--Klein forms in the case of indecomposable Lorentzian symmetric spaces. In this paper, we are interested in Problem \ref{mainprob} for symmetric spaces with signature $(n, 2)$.

For the first step of the problem, we classify four-dimensional indecomposable symmetric spaces $M$ which admit compact Clifford--Klein forms whose transvection group is solvable. Since Kath--Olbrich (\cite{OKCK}) classified Lorentzian spaces which admit compact Clifford--Klein forms, we discuss the space with signature $(2,2)$.

The indecomposable but reducible pseudo-Riemannian symmetric spaces with signature (2,2) are classified by Kath--Olbrich (\cite{OK}). Since a 1-connected pseudo-Riemannian symmetric space $(M,\sigma,g)$ uniquely corresponds to a symmetric triple $(\fg,\sigma,g)$, they classified its symmetric triples.

\begin{fact}[{\cite[Theorem~7.1]{OK}}]\label{Fact:classification}
Let $(\fg,\sigma,g)$ be the symmetric triple corresponding to a 1-connected four-dimensional reducible and indecomposable pseudo-Riemannian symmetric spaces with signature $(2,2)$, and assume that its transvection group is solvable. Then the symmetric triple $(\fg,\sigma,g)$ is isometric to one in the following list.
 \begin{enumerate}
    \renewcommand{\labelenumi}{Case (\Roman{enumi})}
  \item Nilpotent symmetric triples $(\fg_{\text{nil}},\sigma,g_\pm)$ (see Definition \ref{Def:gnil}),
  \item Solvable symmetric triples $(\fg_{D,D'}, \sigma,g)$ (see Definition~\ref{Def:gdd}), where
	\begin{enumerate}
  \item $(D,D')=(\pm\diag{1,\nu},\diag{1,-\nu})\quad (\nu>0)$,\\ $(D,D')=(\pm\diag{1,-\nu},\diag{1,-\nu})\quad (\nu>0,\ \nu\neq 1) $, 
  \item $(D,D')=\left(Q_\nu,Q_{-\nu}\right)\quad (\nu>0)$ (see Notation~\ref{Not:basic}),
  \item $(D,D')=\left(\begin{pmatrix}
	    \pm1&-1\\
	    -1&0
	   \end{pmatrix},
	\begin{pmatrix}
	 0& -1\\ -1&\pm1
	\end{pmatrix}\right),\
	\left(\begin{pmatrix}
	    \pm1&-1\\
	    -1&0
	      \end{pmatrix},
	\begin{pmatrix}
	 0&1\\ 1&\mp1
	\end{pmatrix}\right)$,
  \item $(D,D')=(\pm I_{1,1},I_{1,1})$.
	\end{enumerate}
 \end{enumerate}
\end{fact}



Here, $\fg_{\rm nil}$ is a 3-step nilpotent Lie algebra given in Definition~\ref{Def:gnil}. On the other hand, $\fg_{D,D'}$ denotes an extension of the Heisenberg Lie algebra $\fh_2$ by $
 \begin{pmatrix}
  &D'\\ D&
 \end{pmatrix}\in \fs\fp(2,\R)$ (see Definition~\ref{Def:gdd}). The isometry class of the triple $(\fg_{D,D'},\sigma,g)$ has continuous parameters (see also Figure~4.1).


Our main result in this paper is:

\begin{thm}\label{Thm:mainthm}
Let $(M,\sigma,g)$ be a four-dimensional reducible and indecomposable 1-connected pseudo-Riemannian symmetric space with signature $(2,2)$ whose transvection group is solvable. We denote by $\Gt$ and $\Gi$ its transvection group and isometry group respectively, and set closed subgroups $\Ht\subset \Gt$ and $\Hi\subset \Gi$ by $M\simeq \Gt/\Ht\simeq \Gi/\Hi$.
\begin{enumerate}
 \item $\Gt/\Ht$ admits compact Clifford--Klein forms if and only if its corresponding symmetric triple is $(\fg_{\pm I_{1,1},I_{1,1}},\sigma,g)$.
 \item $\Gi/\Hi$ admits compact Clifford--Klein forms if and only if its corresponding symmetric triple is $(\fg_{\pm I_{1,1},I_{1,1}},g,\sigma)$ or $(\fg_{\rm nil},\sigma,g_\pm)$.
\end{enumerate}
\end{thm}

\begin{table}[h]
 \begin{center}
  \begin{tabular}{c|c|c}
    symmetric triple & $(\fg_{D,D'},\sigma,g)$ & $(\fg_{\rm nil} ,\sigma, g_\pm)$\\\hline
   $\Gt/\Ht$   & $(D,D')=(\pm I_{1,1},I_{1,1})$ & Never \\\hline
   $\Gi/\Hi$ & $(D,D')=(\pm I_{1,1},I_{1,1})$ & Always
  \end{tabular}
   \caption{Symmetric spaces which admit compact Clifford--Klein forms.}
 \end{center}
\end{table}




To prove this theorem, we use two strategies, the constructors (see Definition \ref{Def:constructor}) and intermediate syndetic hulls (see Definition \ref{Def:LSH}). The idea of the constructor for reductive case was introduced by T. Kobayashi (\cite{kob89}), and the following conjecture remains open.
\begin{conj}[{\cite[Conjecture 3.3.10]{KY}}]\label{Conj:Kobayashiconjecture}
 If a homogeneous space $G/H$ of reductive type admits compact Clifford--Klein forms, then $G/H$ admits a reductive constructor, that is, there exists a subgroup $L$ which is reductive in $G$ and acts properly and cocompactly on $G/H$.
\end{conj}

Remark~that the conjecture above does not assert that for a compact Clifford--Klein form $\Gamma\bs G/H$ of a homogeneous space $G/H$ of reductive type, there exists a reductive constructor $L$ containing $\Gamma$ cocompactly.
A Clifford--Klein form $\Gamma\bs G/H$ is {\it standard} (\cite{KK16}) if $\Gamma$ contained in some reductive subgroup $L$ of $G$ acting properly on $X$. In some cases, we obtain a non-standard compact Clifford--Klein form by deforming standard one (see \cite{goldman,kob98}).\par

We show the assumption `reductive type' in this conjecture is crucial by showing a `solvable analogue' of the conjecture does not hold (Example \ref{ex:ceforkobconj}).\\

Organizations of this paper. Section~\ref{Section:Preliminaries} gives basic concepts of pseudo-Riemannian symmetric spaces and Clifford--Klein forms. In Section~\ref{Section:ProperandcocompactinsolvableLiegroups}, we show some general properties about properness and freeness in 1-connected solvable Lie groups. Then we define a class of symmetric spaces $G_{D,D'}/H$ in Section~\ref{Section:Indecomposablesymmetrictriples}. We prove the main theorem in Section~\ref{Section:Proof} using the necessary and sufficient condition for the existence of compact Clifford--Klein forms of $G_{D,D'}/H$ given in Section~\ref{Section:Criterions}. Finally, we show a `solvable analogue' of the Kobayashi's conjecture does not hold in Section~\ref{Section:Onkobayashisconjecture}.

\section{Preliminaries}\label{Section:Preliminaries}

In this section, we review some basic concepts of pseudo-Riemannian symmetric spaces and Clifford--Klein forms.

\subsection{Notation}

In this subsection, we prepare notation used in this paper.

\begin{notation}\label{Not:basic}
 \begin{itemize}
  \item $\R^\times$ := $\R-\{0\}$,
  \item $e$ : the identity element of a group,
  \item $\A_g$ : the inner automorphism with respect to an element $g$ of a group,
  \item $\A_S H:=\set{\A_s h}{s\in S, h\in H}$ for subsets $S, H$ of a group,
  \item $\mathcal{Z}_G$ : the center of a group $G$,
  \item $\mathcal{Z}_G(g)$ : the centralizer of an element $g\in G$ in a group $G$,
  \item $\mathcal{N}_G(L)$ : the normalizer of a subgroup $L\subset G$ in a group $G$,
  \item $\Der \fg$ : the derivation algebra of a Lie algebra $\fg$,
  \item $I_{p,q}:=\begin{pmatrix}
		   I_p&\\
		   &-I_q
		  \end{pmatrix}\in GL(p+q,\R)$,
  \item $M^T$ : the transposed matrix of a matrix $M$,
  \item $\diag{a_i}:=\begin{pmatrix}
		     a_1&&&\\
		     &a_2&&\\
		     &&\ddots&\\
		     &&&a_n
		    \end{pmatrix}$,
  \item $\sym(n,\R):= \set{M\in M(n,\R)}{M\text{ is symmetric}}$,
  \item $\sym^{(\text{reg})} (n,\R):=\set{M\in \sym(n,\R)}{\det M\neq 0}$,
  \item $Q_\nu:=
	\begin{pmatrix}
	 \nu&1\\1&-\nu
	\end{pmatrix}\in M(2,\R)$.
 \end{itemize}
\end{notation}

In this paper, we use the terminology {\it inner product} as a non-degenerate symmetric bilinear form (not necessarily positive definite) and Lie algebras are real and finite dimensional.

\subsection{Symmetric triples and pseudo-Riemannian symmetric spaces}

In this subsection, we recall a correspondence between 1-connected pseudo-Riemannian symmetric spaces and symmetric triples.

\begin{defi}[metric Lie algebra with involution, \cite{CW,OK}]
Let $\fg$ be a Lie algebra, $\sigma$ an involution on $\fg$ and $g$ an (indefinite) inner product on $\fg$. We say $(\fg, \sigma, g)$ is a {\it metric Lie algebra with involution} if $\fg$, $\sigma$ and $g$ are mutually compatible, that is, satisfy the following conditions:
\begin{enumerate}
 \item the inner product $g$ is $\sigma$-invariant,
 \item the inner product $g$ is $\fg$-invariant, namely, 
       \[
	g([X,Y],Z)+g(Y,[X,Z])=0 \quad (\forall X, Y, Z\in \fg).
       \]
\end{enumerate}
\end{defi}

\begin{defi}[symmetric triple, \cite{CW,OK}]\label{Def:symmetrictriple}
 A metric Lie algebra with involution $(\fg,\sigma, g)$ is called a {\it symmetric triple} if the subspace $\fq\coloneqq\fg^{-\sigma}$ satisfies $[\fq,\fq]=\fg^\sigma$.
\end{defi}

\begin{defi}[homomorphism on symmetric triple, \cite{CW,OK}]\label{Def:hom}
 For two symmetric triples $(\fg_1,\sigma_1,g_1)$ and $(\fg_2,\sigma_2,g_2)$, a Lie algebra homomorphism $\phi:\fg_1 \rightarrow \fg_2$ is said to be a {\it homomorphism of symmetric triple} if $\phi$ is compatible with the involutions and the inner products, that is, satisfies the following conditions:
 \begin{enumerate}
  \item $\sigma_2\circ \phi=\phi\circ\sigma_1$,
  \item $g_2(\phi(X_1),\phi(X_2))=g_1(X_1,X_2)$ \quad($\forall X_1,X_2\in\fg_1$).
 \end{enumerate}
\end{defi}

\begin{note}\label{Note:qtosigma}
Let $\fg$ be a Lie algebra and $\fh$ its subalgebra. The following correspondence is bijective:
\begin{align*}
 &\{\text{involutions $\sigma$ on $\fg$ satisfying $\fg^\sigma=\fh$}\} \\
 &\simeq\{\text{complementary spaces $\fq\subset \fg$ of $\fh$ satisfying $[\fq,\fh]\subset \fq$ and $[\fq,\fq]\subset \fh$}\} \\
 &\sigma \mapsto \fg^{-\sigma}.
\end{align*}
\end{note}

\begin{fact}[\cite{CP}]\label{Fact:gonq}
 Let $\fg$ be a Lie algebra and $\sigma$ its involution. Put $\fh:=\fg^\sigma$ and $\fq:=\fg^{-\sigma}$. If $[\fq,\fq]=\fh$, then the following restriction is bijective:
 \begin{align*}
  \{\text{$\fg$-invariant inner product on $\fg$}\}&\simeq \{\text{$\fh$-invariant inner product on $\fq$}\}\\
  g&\mapsto g|_{\fq\times \fq}.
 \end{align*} 
Especially, any $\fg$-invariant inner product is also $\sigma$-invariant in this case.
\end{fact}

\begin{note}\label{Note:kousei}
For a Lie algebra $\fg$ and its subalgebra $\fh$, assume a subspace $\fq\subset \fg$ and an inner product $g$ on $\fq$ satisfy the following conditions:
 \begin{itemize}
  \item $\fg=\fq\oplus\fh$,
  \item $[\fq,\fh]\subset \fq$ and $[\fq,\fq]=\fh$,
  \item $g$ is $\fh$-invariant.
 \end{itemize}
Then a symmetric triple $(\fg,\sigma,g)$ is uniquely determined by $g$ and $\fq$ (see Note~$\ref{Note:qtosigma}$ and Fact~$\ref{Fact:gonq}$).
\end{note}

\begin{defi}[\cite{CW,OK}]
 For a symmetric triple $(\fg,\sigma,g)$, we call the signature of $g$ (on $\fq$) {\it the signature of the symmetric triple}. 
\end{defi}

In the following, we review the correspondence between symmetric triples and pseudo-Riemannian symmetric spaces.

\begin{fact}[{\cite[Ch.I Section 2]{CP},\ \cite{OK}}]\label{Fact:transvection}
 There is a bijection between the isomorphic classes of $1$-connected pseudo-Riemannian symmetric spaces with signature $(p, q)$ and the isomorphic classes of symmetric triples with signature $(p,q)$. Let $(\fg, \sigma, g)$ be a symmetric triple and $M$ its corresponding pseudo-Riemannian symmetric space, then $\fg$ is the Lie algebra of the transvection group of $M$. The 1-connected Lie group $G$ with Lie algebra $\fg$ is the transvection group of $M$.
\end{fact}


Like the case of Riemannian symmetric spaces, the goal of the classification problem of pseudo-Riemannian symmetric spaces is to classify their `minimum units', which are indecomposable.
For Riemannian spaces, they are irreducible symmetric spaces, but are not necessarily for pseudo-Riemannian spaces. We define reducibilities and decomposabilities of symmetric triples and symmetric spaces.

\begin{defi}[\cite{N54}]
 We say a symmetric triple $(\fg,\sigma,g)$ is \emph{reducible} if the isotropy representation $\ad:\fh\to \fg\fl(\fq)$ is reducible for $\fh:=\fg^\sigma$ and $\fq:=\fg^{-\sigma}$. A 1-connected pseudo-Riemannian symmetric space $G/H$ is said to be \emph{reducible} if its triple is reducible.
\end{defi}

\begin{defi}[\cite{CW,OK}]
 For two symmetric triples $\ft_1=(\fg_1,\sigma_1,g_1)$ and $\ft_2=(\fg_2,\sigma_2,g_2)$, the triple $\ft_1\oplus\ft_2 \coloneqq (\fg_1\oplus\fg_2,\sigma_1\oplus \sigma_2,g_1\oplus g_2)$ is also a symmetric triple. We say $\ft_1\oplus\ft_2$ is the {\it direct sum} of $\ft_1$ and $\ft_2$. A symmetric triple is said to be {\it decomposable} if it is written as the direct sum of two non-trivial symmetric triples.
\end{defi}

\begin{defi}[\cite{CW,OK}]
A pseudo-Riemannian symmetric space is said to be {\it decomposable} if the space is isomorphic to the direct product of two non-trivial pseudo-Riemannian symmetric spaces.
\end{defi}


The decomposability of pseudo-Riemannian symmetric spaces corresponds to that of symmetric triples.

\begin{prop}[{\cite[Proposition~4.4]{CP}}]
 Let $M$ be a 1-connected pseudo-Riemannian symmetric space, and $\ft \coloneqq (\fg,\sigma, g)$ the corresponding symmetric triple. Then the following correspondence is one to one.
 \[
  \{\text{decompositions of $\ft$}\}\to \{\text{decompositions of $M$}\}\quad \ft_1\oplus\ft_2\mapsto M_1\times M_2,
 \]
 where $M_1$ and $M_2$ are the corresponding 1-connected pseudo-Riemannian symmetric spaces of symmetric triples $\ft_1$ and $\ft_2$, respectively.
\end{prop}

\subsection{Isometry groups}

In this subsection, we prepare some lemmas which are used for calculating isometry group of pseudo-Riemannian symmetric spaces. For this, we use:

\begin{notation}
For a 1-connected pseudo-Riemannian symmetric space $(M,S,g)$, we set the origin point $o\in M$. We put
 \begin{itemize}
  \item $\Gi$ : the isometry group,\ $\Hi$ : the stabilizer of $o$,
  \item $\Gt$ : the transvection group,\ $\Ht$ : the stabilizer of $o$,
  \item $\fgi=\fhi\oplus\fqi$ : the Lie algebra of $\Gi$ and its eigenspace deconposition with respect to $\Ad_{S_0}$.
  \item $\fgt=\fht\oplus \fqt$ : the Lie algebra of $\Gt$ and its eigenspace deconposition with respect to $\Ad_{S_0}$.
 \end{itemize}
\end{notation}

\begin{lem}\label{Lem:transvection}
 For a 1-connected pseudo Riemannian symmetric space $(M,S,g)$, the map $\Hi\ltimes \Gt\to\Gi,\ (h,g)\mapsto gh$ is surjective group homomorphism and its kernel is:
 \[
  \wt{\Ht}:=\set{(g,g\inv)\in \Hi\ltimes \Gt}{g\in\Ht}.
 \]
Especially, we have a Lie group isomorphism $\Gi\simeq (\Hi\ltimes\Gt)/\wt{\Ht}$. Moreover, if there exists a closed subgroup $T\subset \Hi$ satisfying $\Hi\simeq T\ltimes \Ht$, we have a Lie group isomorphism $\Gi\simeq T\ltimes \Gt$.
\end{lem}

\begin{proof}
Since the $\Gt$-action on $M$ is transitive, for any $g\in \Gi$, there exists $g'\in \Gt$ satisfying ${g'}\inv g\in\Hi$. Therefore, the map $\Hi\ltimes \Gt\to\Gi,\ (h,g)\mapsto gh$ is surjective group homomorphism. The latter statement follows from the following Note with $A:=T,\ B:=\Ht$ and $C:=\Gt$.
\end{proof}

\begin{note}
 Let $A,B$ and $C$ be Lie groups. Assume that there exists a map $\phi:A\to \Aut(C)$ such that $B\subset C$ is a $\phi$-invariant closed subgroup. We set $\psi:A\ltimes B\to \Aut(C)$ by $\psi(a,b)(c)=\A_b\phi_a(c)$. Then $\psi$ is a Lie group homomorphism. Moreover the following map $f$ is also a Lie group homomorphism and its kernel is $\set{(e,b,b\inv)}{b\in B}=:\wt{B}$.
 \[
  f:(A\ltimes B)\ltimes_{\psi}C\to A\ltimes C,\quad (a,b,c)\mapsto (a,cb).
 \]
 Especially, we have $((A\ltimes B)\ltimes C)/\wt{B}\simeq A\ltimes C$.
\end{note}

\begin{lem}\label{Lem:hishape}
Let $(\fgt,\sigma,g)$ is the symmetric triple which corresponds to 1-connected pseudo-Riemannian symmetric space $(M,S,g)$. Then the map  $\Phi:\Hi\to \Aut(\fgt,\sigma,g),\quad h\mapsto \Ad_h|_{\fgt}$ is a Lie group isomorphism.
\end{lem}

\begin{proof}
First, we check that $\Phi$ is well-defined. Since $\Gt\subset \Gi$ is a normal subgroup, we have $\A_h(\Gt)\subset \Gt$. By $\Hi\subset \Gi^\sigma$, we have $\A_{S_0}h=h$. Therefore $\Ad_{S_0}$ and $\Ad_h$ are commutative, namely, $\Ad_{h}$ preserves the decomposition $\fgi=\fhi\oplus \fqi$ and $\fgt=\fht\oplus\fqt$. On the other hand, $\Ad_h$ preserves the inner product on $\fqi=\fqt$. Since it is commutative with $\sigma$ and a Lie algebra homomorphism, it also preserves the inner product on $\fht$, and so $g$ on $\fgt$. Then we have $\Ad_h|_{\fgt}\in \Aut(\fgt,\sigma,g)$.\par
We set a map $\Psi:\Aut(\fgt,\sigma,g)\to \Hi,\ A\mapsto \Psi(A):=\ol{f_A}$ as follows. Let $(\Gt,\Ht,\sigma,g)$ be the symmetric triple which corresponds to the symmetric triple $(\fgt,\sigma,g)$. Put $f_A:\Gt\to\Gt$ the induced Lie group homomorphism by the Lie group homomorphism $A:\fgt\to\fgt$.  Note that $f_A$ is a Lie group isomorphism. Since $f_A$ preserves $\Ht$, it induces $\ol{f_A}:\Gt/\Ht\to \Gt/\Ht$, and it is an isometry by Note~\ref{Note:isom}. Therefore we have $\ol{f_A}\in \Hi$.\par
Next we check that $\Psi$ is the inverse map of $\Phi$. First we show $\Psi\circ \Phi=\id$. Let $h\in \Hi$. It is enough to show $\Psi(\Phi(h))=\ol{f_{\Ad_h}}$ is coincides to $h$. It follows from $f_{\Ad_h}=\A_h$, and $h:M\to M$ and $\ol{f_{\Ad_h}}:M\to M$ make the following diagram commutative:
\[
  \xymatrix{
    \Gt \ar[r]^{f_{\Ad_h}} \ar[d]_{\pi} & \Gt \ar[d]_{\pi} \\
    M \ar[r] & M.
  }
\] 
\end{proof}
 Then we check $\Phi\circ \Psi=\id$. Let $A\in\Aut(\fgt,\sigma,g)$ and $\ol{f_A}\in \Hi$. We are enough to show $\Ad_{\ol{f_A}}=A$, which is equivalent to $\Ad_{\ol{f_A}}|_{\fqt}=A|_{\fqt}$. For $X\in\fq$, we have: 
 \[
  A(X)=d(\ol{f_A})_{eH}(X)=\dfrac{d}{dt}\Big|_{t=0}\ol{f_A}(\exp tX)\Hi=\Ad_{\ol{f_A}}(X).
 \]
Here, $A(X)=d(\ol{f_A})$ follows from Note~\ref{Note:gentenbibun}.

\begin{note}\label{Note:isom}
 Let $G$ be a Lie group, $H$ its closed subgroup and $g$ a $G$-invariant inner product on $G/H$. Suppose that a group isomorphism $f:G\to G$ preserves $H$ and put $\ol{f}:G/H\to G/H$ the induced map. If $d\ol{f}_{eH}$ preserves inner product $g_{eH}$ on $T_{eH}(G/H)$, then $\ol{f}$ is an isometry.
\end{note}

\begin{note}\label{Note:gentenbibun}
 Let $f$ be an automorphism on $(G,H,\sigma)$. Put $\fg:=\Lie{G}$ and $\fq:=\fg^\sigma$. Then we have $d\ol{f}|_{eH}=df_e|_{\fq}$, where $\ol{f}:G/H\to G/H$ is the induced map by $f$ and we identify $T_{eH}G/H\simeq \fq$.
\end{note}

\begin{note}\label{Note:ANlem}
 Let $G$ be a Lie group, $A\subset G$ a closed subgroup and $N\subset G$ a normal subgroup. If $G=AN$ and $A\cap N=\{e\}$, then we have $G\simeq A\ltimes N$, where $A$ acts on $N$ as an inner automorphism.
\end{note}

\subsection{Clifford--Klein forms}

In this subsection, we review Clifford--Klein forms following \cite{kob89} and \cite{kob97}. 

\begin{defi}[\cite{kob89}]\label{Def:ckdefi}
 Let $G$ be a Lie group, $H$ its closed subgroup, and $\Gamma$ its discrete subgroup. Assume the $\Gamma$-action on $G/H$ is (fixed point) free and properly discontinuous. Then the quotient space  $\Gamma\bs G/H$ has the unique manifold structure such that the natural surjection $\pi: G/H\to \Gamma \bs G/H$ is a $C^\infty$-covering map. The manifold $\Gamma\bs G/H$ is said to be a {\it Clifford--Klein form} of $G/H$.
\end{defi}

In the study of Clifford--Klein forms, Problem \ref{Prob:cptckform} is a significant open question. Let us recall basic terminologies for Problem \ref{Prob:cptckform}.

\begin{defi}[\cite{kob89}]
 Suppose a locally compact group $L$ acts on a locally compact space $X$. The $L$-action is said to be {\it proper} if $\set{\ell\in L}{\ell S\cap S\neq \emptyset}$ is compact for any compact subset $S\subset X$.
\end{defi}

It is easy to check the following:

\begin{note}
 In the setting of Definition~\ref{Def:ckdefi}, if the $L$-action on $X$ is proper, any $L$-orbit is closed in $X$.
\end{note}

\begin{fact}[\cite{kob89}]\label{Fact:contianalogy}
 Let $L$ be a locally compact group and $X$ a locally compact space. Assume $L$ acts on $X$ and $\Gamma$ is a uniform lattice (cocompact discrete subgroup) of $L$. Then the following statements hold. 
 \begin{enumerate}
  \item The $\Gamma$-action on $X$ is cocompact if and only if so is the $L$-action.
  \item The $\Gamma$-action on $X$ is properly discontinuous if and only if the $L$-action is proper.
 \end{enumerate}
\end{fact}

We recall some definitions and properties.

\begin{defi}[{\cite[Definition 6]{kob92b}, \cite[Definition 2.1.1]{kob96}}]
 Let $G$ be a locally compact group, and $L$ and $H$ its subsets.
 \begin{enumerate}
  \item We say the pair $(L, H)$ is {\it proper} in $G$, denoted by $L\pitchfork H$ in $G$, if the set $L\cap SHS$ is relatively compact in $G$ for any compact set $S\subset G$. 
  \item We say the pair $(L, H)$ has {\it the property} (CI) in $G$, if the set $L\cap gHg\inv$ is relatively compact in $G$ for any $g\in G$. 
  \item We say the pair $(L, H)$ is {\it free} if the condition $L\cap gHg\inv =\{e\}$ holds for any $g\in G$. 
  \item We denote by $L\sim H$ in $G$ the existence of a compact set $S\subset G$ satisfying $L\subset SHS$ and $H\subset SLS$.
 \end{enumerate}
\end{defi}

\begin{rem}
 In \cite{kob92b}, Kobayashi defined the property (CI) for subgroups $L$ and $H$, but here we define it for subsets $L$ and $H$ for the sake of Lemma~\ref{Lem:closedcone}.
\end{rem}

\begin{property}[\cite{kob92b,kob96}]\label{Property:basicproperty1}
 Let $G$ be a locally compact group, and $H, H'$ and $L$ its subsets.
 \begin{enumerate}
  \item The pair $(L, H)$ is proper (resp. has the property (CI), is free) in $G$ if and only if so is $(H, L)$ in $G$. 
  \item The relation $\sim$ is an equivalence relation. 
  \item If  $H\sim H'$ in $G$, then  $L\pitchfork H$ if and only if $L\pitchfork H'$ in $G$.
  \item If the pair $(L, H)$ is proper, then $(L, H)$ has the property (CI).
 \end{enumerate}
\end{property}

\begin{property}[{\cite[Observation 2.13]{kob96}}]\label{Property:basicproperty2}
 Let $G$ be a locally compact group, and $L$ and $H$ its closed subgroups.
 \begin{enumerate}
    \item The $L$-action on $G/H$ is proper if and only if $L \pitchfork H$ in $G$.
  \item The $L$-action on $G/H$ is free if and only if the pair $(L, H)$ is free in $G$.
 \end{enumerate}
\end{property}

\begin{property}\label{Property:centerquotient}
 Let $G$ be a locally compact group and $N$ its closed normal subgroup, and $H$, $L$ closed subgroups of $G$. We denote by $\widetilde{G}$, $\widetilde{H}$ and $\widetilde{L}$ the image of $G$, $L$ and $H$, respectively, by the natural projection $\pi:G\to G/N$. Then we have:
 \begin{enumerate}
  \item If $N/(L \cap N)$ is compact and $L$ is discrete, $\widetilde{L}\subset \widetilde{G}$ is also discrete.
  \item If $N/(L \cap N)$ is compact, the condition $L\pitchfork H$ in $G$ implies $\widetilde{L}\pitchfork \widetilde{H}$ in $\widetilde{G}$.
  \item If $L \bs G/H$ is compact, so is $\widetilde{L}\bs {\widetilde{G}}/\widetilde{H}$.
 \end{enumerate} 
\end{property}

The statement (1) follows from \cite[Lemma~5.1.4]{repofnilp}. The statement (2) follows from \cite[Lemma~1.3(2)]{kob93}. The statement (3) is easy.

Finally, we prepare some easy lemmas which are used in Sections~\ref{Section:ProperandcocompactinsolvableLiegroups} and \ref{Section:Criterions}.

\begin{lem}\label{Lem:closedcone}
 Let $C_1$ and $C_2$ be two closed cones in $\R^n$. Then the following conditions are equivalent:
 \begin{eqenum}
  \item the pair $(C_1,C_2)$ has the (CI) property, namely, $C_1\cap C_2=\{0\}$,
  \item the pair $(C_1,C_2)$ is proper in $\R^n$.
 \end{eqenum}
\end{lem}

\begin{proof}
 Since the implication (b)$\imply$(a) is easy, we prove the implication (a)$\imply$(b). We take any $R\in\R_{>0}$ and denote by $B(R)$ the closed ball in $\R^n$. It is enough to show that $(C_1+B(R))\cap C_2$ is relatively compact.
 Set $d_0\in\R_{>0}$ as the distance between $C_1$ and $C_2\cap S^{n-1}$, where $S^{n-1}$ is the unit sphere in $\R^n$. Take any $x\in (C_1+B(R))\cap C_2$, then we have $R\geq d(C_1,x)\geq \|x\|d_0$, and so $(C_1+B(R))\cap C_2\subset B(R/d_0)$.
\end{proof}

\begin{lem}\label{Lem:properlem}
 Let $G$ be a locally compact group and $N$ its closed normal subgroup. Let $L_0$ and $H$ be subsets of $N$, and $L_1\neq \emptyset $ a subset of $G$ satisfying $L_1\pitchfork N$ in $G$. Set $L:=L_1L_0$, then the following conditions are equivalent:
 \begin{eqenum}
  \item $(\A_{S}L_0)\pitchfork H$ in $N$ for any compact set $S\subset G$, 
  \item $L_0\pitchfork H$ in $G$,
  \item $L\pitchfork H$ in $G$.
 \end{eqenum}
\end{lem}

\begin{proof}
 Since the implications (c)$\imply$(b)$\imply$(a) are easy, we prove the implication (a)$\imply$(c). Let $S\subset G$ be a compact set. We have:
 \[
  SLS\inv \cap H\subset SL_1S\inv (\A_SL_0)\cap H=(SL_1 S\inv \cap N) \A_S L_0\cap H.
 \]
 By the assumption $L_1\pitchfork N$ in $G$, we take a compact set $K\subset G$ satisfying $SL_1S\inv \cap N\subset K$. Then we have:
 \[
  (SL_1S\inv \cap N)\A_SL_0\cap H \subset K (\A_SL_0)\cap H.
 \]
 By the condition (a), the subset $K (\A_SL_0)\cap H$ is relatively compact in $G$. Therefore the condition (c) follows.
\end{proof}

\begin{lem}\label{Lem:solvproper}
 Let $L$ and $N$ be locally compact groups. Assume $L$ acts on $N$ continuously as group automorphisms. Put $G:=L\ltimes N$, then we have $L\pitchfork N$ in $G$. 
\end{lem}

\begin{proof}
 It is enough to show $L\cap (S_1\times S_2)N(S_1\times S_2)$ is relatively compact for any compact subsets $S_1\subset L$ and $S_2\subset N$. This follows from:
 \[
  L\cap (S_1\times S_2)N(S_1\times S_2)\subset L\cap S_1S_1=S_1S_1.
 \]
\end{proof}

\begin{note}\label{Note:component}
 Let $G$ be a Lie group with finite connected component and $H$ its closed subgroup. Put $G_0\subset G$ the identity component. Put $H_0:=G_0\cap H$. Then $G/H$ admits compact Clifford--Klein forms if and only if so does $G_0/H_0$.
\end{note}

\subsection{Constructors}

By Fact~\ref{Fact:contianalogy}, it is natural to define the following subgroups called constructors, of which the terminology is introduced in \cite{KY07}\footnote{see also \url{http://coe.math.sci.hokudai.ac.jp/sympo/ccyr/2006/pdf/TaroYOSHINO.pdf}}. In this subsection, we define constructors and see some basic properties.

\begin{defi}[{\cite{KY07}}]\label{Def:constructor}
 Let $G/H$ be a homogeneous space of a Lie group $G$. A closed and connected subgroup $L\subset G$ is said to be a {\it constructor} of $G/H$ if the natural action of $L$ on $G/H$ is proper and cocompact.
\end{defi}

We think constructors for homogeneous spaces of solvable type. We note:

\begin{fact}[{\cite{Hochschild}}]
 A connected subgroup of a 1-connected solvable Lie group is closed. 
\end{fact}

It is important to consider the existence of a constructor for the existence problem of compact Clifford--Klein forms.

\begin{defi}[\cite{witte}]
 Let $G$ be a Lie group and $\Gamma$ its closed subgroup. We say a closed and connected subgroup $L\subset G$ a {\it syndetic hull} of $\Gamma$ if $L$ includes $\Gamma$ cocompactly.
\end{defi}

\begin{fact}[\cite{saito,BK}]\label{Fact:solvablesyndetichull}
 Let $G$ be a 1-connected completely solvable Lie group and $\Gamma\subset G$ a closed subgroup. Then there exists a unique syndetic hull $L$ of $\Gamma$. Especially, if the space $G/H$ has a compact Clifford--Klein form $\Gamma\bs G/H$, the space $G/H$ has a constructor $L$ which is the syndetic hull of $\Gamma$.
\end{fact}

\begin{rem}
 \label{Rem:kobconj}
 The assumption of complete solvability in the above fact is crucial. In fact, a solvable Lie group $G$ may have a discrete subgroup $\Gamma$ without its syndetic hulls (see Example \ref{ex:ceforkobconj}).
\end{rem}

\section{Properness and cocompactness in solvable Lie groups}\label{Section:ProperandcocompactinsolvableLiegroups}

In this section, we show some criterions to check properness and cocompactness in 1-connected solvable Lie groups, which are used to show the main theorem. The main results in this section are Propositions \ref{Prop:solvablecri}, \ref{Prop:propciequiv} and \ref{Prop:criofccp}.

\subsection{Freeness and the property (CI) in solvable Lie groups}

In this subsection, we review some criterions for freeness and the property (CI) in solvable Lie groups.

First, the following note gives a criterion of the property (CI) for 1-connected nilpotent Lie groups.

\begin{note}[\cite{KN}]\label{Note:nilpci}
 Let $G$ be a 1-connected nilpotent Lie group, and $L$ and $H$ its connected subgroups. Then the following conditions are equivalent:
 \begin{eqenum}
  \item The pair $(L,H)$ has the property (CI),
  \item $\Ad_G \fl\cap \fh =\{0\}$,
  \item $\fl\cap\Ad_G \fh=\{0\}$.
 \end{eqenum}
 Here, $\fl$ and $\fh$ are the Lie algebras of $L$ and $H$, respectively.
\end{note}

This note is easily shown by using the diffeomorphism $\exp:\fg\to G$. It is easy to show the following note in the same way.

\begin{note}\label{Note:nilpci2}
 Note~\ref{Note:nilpci} also holds under the assumptions that $G$ is an arbitrary Lie group and there exists a 1-connected closed normal nilpotent subgroup $N\subset G$ satisfying $L,H\subset N$.
\end{note}

Next, we review the following:

\begin{fact}[{\cite[Theorem~2.3]{Hochschild}}]\label{Fact:keyforcifree}
 A compact subgroup of a 1-connected solvable Lie group is trivial.
\end{fact}

We have two corollaries from this fact.

\begin{cor}\label{Cor:solvablefree}
 Let $G$ be a 1-connected solvable Lie group, and $L$ and $H$ closed subgroups of $G$. Then the following conditions are equivalent.
 \begin{eqenum}
  \item The pair $(L, H)$ is free in $G$.
  \item The pair $(L, H)$ has the property (CI) in $G$.
 \end{eqenum}
\end{cor}

\begin{rem}
 For a 1-connected exponential solvable Lie group $G$, the above statement was proven by Baklouti and K\'edim \cite{BK}.
\end{rem}

By using Corollary~\ref{Cor:solvablefree} and Property \ref{Property:basicproperty1}, we have:

\begin{note}\label{note:properisfree} 
Let $G$ be a 1-connected solvable Lie group, and $L$ and $H$ its closed subgroups.
 \begin{enumerate}
  \item If the pair $(L, H)$ is proper, the quotient space $L\bs G/H$ has a manifold structure. 
  \item Let $\Gamma \subset L$ be a uniform lattice. Assume the action $L\acts G/H$ is proper and cocompact, then $\Gamma\bs G/H$ is a compact Clifford--Klein form.
 \end{enumerate}
\end{note}

\subsection{Constructors in solvable homogeneous space}

In this subsection, we show some propositions for the existence of constructors in solvable homogeneous spaces. First, we see a criterion of the cocompactness of the $L$-action.

\begin{prop}\label{Prop:solvablecri}
 Let $G$ be a 1-connected solvable Lie group, and $L$ and $H$ its connected subgroups. Assume the $L$-action on $G/H$ is proper. Then the following conditions are equivalent:
 \begin{eqenum}
  \item the space $L\bs G/H$ is compact, 
  \item $G=LH$,
  \item $\fg=\fl\oplus\fh$ as a linear space.
 \end{eqenum}
 Here, $\fg$, $\fh$ and $\fl$ are Lie algebras of $G, H$ and $L$, respectively.
\end{prop}

\begin{proof}
Since the implications (b)$\imply$(a) is clear, we first show the implication (a)$\imply$(b). The condition (b) is equivalent to the transitivity of the $L$-action on $G/H$, so we are enough to show that the space $L\bs G/H$ consists of one point. Since $G$ is a 1-connected solvable Lie group, and $H$ and $L$ are connected subgroups, $G/H$ and $L$ are contractible by Note~\ref{Note:1conn} and Lemma~\ref{Lem:contractiblelemma} below.
By Note~\ref{note:properisfree} (1), the quotient space $L\bs G/H$ has a manifold structure, it is one point by Lemma~\ref{Lem:contractiblequotient}.
 Next, we show the implication (b)$\imply$(c). Since $G/H$ is an $L$-orbit, we have $\dim(G/H)\leq \dim L$. On the other hand, by the properness of the $L$-action we have $\fl\cap \fh=\{0\}$ (Note~\ref{Note:1connproper}). Then we have $\dim G\geq \dim H+\dim L$ and so we obtain $\dim G=\dim H +\dim L$, which implies $\fg =\fl\oplus\fh$.\par
 Finally, we check the implication (c)$\imply$(b). We consider the $L$-orbit of the origin point $eH$. The orbit is closed since the $L$-action on $G/H$ is proper. On the other hand, since the $L$-action is free  by Corollary~\ref{Cor:solvablefree}, the dimension of $L$-orbit equals to $\dim L=\dim(G/H)$. Hence the $L$-orbit is open. Since $G/H$ is connected, $G/H$ coincides with the $L$-orbit.
\end{proof}

\begin{note}[\cite{bourbaki}]\label{Note:1conn}
 Any 1-connected solvable Lie group is diffeomorphic to a Euclidian space.
\end{note}

\begin{note}\label{Note:1connproper}
 Let $H$ and $L$ be closed subgroups of a 1-connected solvable Lie group $G$. Assume $L\pitchfork H$ in $G$, then we have $\fl\cap \fh=\{0\}$.
\end{note}

\begin{lem}\label{Lem:contractiblequotient}
 Suppose a contractible Lie group $G$ acts on a contractible manifold $M$. If the quotient space $G\bs M$ is a compact manifold, it consists of one point.
\end{lem}

This lemma is an immediate consequence of the following two lemmas.
\begin{lem}\label{Lem:contractiblelemma}
 Let $G$ be a contractible topological group and act on a contractible space $M$. Then $G\bs M$ is contractible.
\end{lem}

\begin{proof}
 By the homotopy exact sequence of the fiber bundle $G\to M\to G\bs M$, we have $\pi_{i}(G\bs M)=0\ (\forall i\in \N)$. Then we have $G\bs M$ is contractible by J. H. C. Whitehead's theorem.
\end{proof}

\begin{lem}
 A contractible and closed manifold consists of one point.
\end{lem}

\begin{proof}
 Let $M$ be a contractible and closed manifold. Since an arbitrary vector bundle over $M$ is trivial, $M$ is orientable. For a volume element $\omega$ of $M$, we have $\int_M\omega=0$ by Stokes' theorem.
\end{proof}




\subsection{Properness, the property (CI) and cocompactness}

To check the property (CI) is easier than the properness. The property (CI) was introduced by T. Kobayashi and the equivalence of properness and the property (CI) was shown for any pair of closed reductive subgroups of linear reductive Lie groups \cite{kob92b}. Lipsman considered an extension of Kobayashi's theory to non-reductive case \cite{lipsman}. For 1-connected nilpotent Lie groups, the equivalence of the properties is known as Lipsman's conjecture. About this conjecture, the following results have been obtained so far. The properness and the property (CI) are equivalent for less than or equal to 3-step nilpotent Lie groups \cite{nasrin,3step,BK05} and not necessarily equivalent for 4-step nilpotent Lie groups \cite{3step}. In this subsection, we generalize the following Nasrin's result (Fact~\ref{Fact:nasrin}) in Proposition~\ref{Prop:propciequiv} and introduce a criterion of cocompactness in a similar setting in Proposition~\ref{Prop:criofccp}.

\begin{fact}[\cite{nasrin} for 2-step nilpotent Lie groups, \cite{3step,BK05} for 3-step nilpotent Lie groups]\label{Fact:nasrin}
 Let $G$ be a 1-connected 3-step nilpotent Lie group, and $L$ and $H$ its connected subgroups. Then $L\pitchfork H$ in $G$ if and only if the pair $(L, H)$ has the (CI) property in $G$.
\end{fact}


\begin{setting}\label{set:propersetting}
Let $G$ be a Lie group, and $N$ its closed normal subgroup. Assume $N$ is 1-connected nilpotent. Let $L_0$ and $H$ be connected subgroups of $N$, and $L_1$ a closed subgroup of $\mathcal{N}_G(L_0)$ (see Notation~\ref{Not:basic}) satisfying $L_1\pitchfork N$ in $G$. Set $L:=L_1L_0=L_0L_1$.
\end{setting}

\begin{prop}\label{Prop:propciequiv}
Under Setting~\ref{set:propersetting}, we additionally assume $N$ is 2-step nilpotent. Then the following conditions are equivalent:
\begin{eqenum}
 \item the pair $(L_0,H)$ has the property (CI) in $G$,
 \item $L\pitchfork H\ \text{in}\ G$.
\end{eqenum}
\end{prop}

\begin{proof}
It is enough to show that the following four conditions are equivalent:
 \begin{enum}{\roman}
  \item the pair $(L_0,H)$ has the property (CI) in $G$,
  \item $\Ad_G\fl_{0}\cap\fh=\{0\}$,
  \item $(\A_{S}L_0)\pitchfork H$ in $N$ for any compact set $S\subset G$,
  \item $L\pitchfork H$ in $G$.
 \end{enum}
Here, $\fl_0$, $\fh$ and $\fn$ are the Lie algebras of $L_0$, $H$ and $N$, respectively.
Since the exponential map $\exp:\fn\to N$ is diffeomorphism, we denote by $\log$ its inverse.\par
The implication (i)$\imply$(ii) comes from Note~\ref{Note:nilpci2}, the implication (iii)$\imply$(iv) holds by Lemma~\ref{Lem:properlem}, and the implication (iv)$\imply$(i) follows from Property~\ref{Property:basicproperty1}(4).\par
 Then we show the implication (ii)$\imply$(iii). Take any compact sets $S\subset G$ and $T\subset \fn$. It is enough to show that the subset $(\exp T(\A_S L_0)\exp T)\cap H$ is compact. Since $N$ is 2-step nilpotent, for $X\in \fn$ we have:
 \begin{align*}
  \exp T \exp X \exp T&\subset \exp\left(\left(\id +\dfrac{1}{2}\ad(T-T)\right)X+T+T+\dfrac{1}{2}[T,T]\right)\\
  &= \exp (\Ad_{S'} X+T'),
 \end{align*} 
 where $S':=\exp((T-T)/2)\subset N$ and $T':= T+T+[T,T]/2\subset\fn$.
  Then we have:
 \[
  \log(\exp T (\A_{S}L_0)\exp T)\subset \Ad_{S'}\Ad_S\fl_0+T'=\Ad_{S' S}\fl_0+T'=C+T',
 \]
where $C:=\Ad_{S' S}\fl_0$ is a closed cone. On the other hand, by the condition~(ii), we have $C\cap \fh=\{0\}$. Hence the pair of the closed cones $\left(C, \fh\right)$ is proper by Lemma~\ref{Lem:closedcone}. Then the subset $\left(C+T'\right)\cap \fh$ is compact and so is the subset $(\exp T(\A_S L_0)\exp T)\cap H$.
\end{proof}

Finally, we introduce criterion of cocompactness under a similar situation in Proposition~\ref{Prop:propciequiv}.

\begin{prop}\label{Prop:criofccp}
 Under Setting \ref{set:propersetting}, assume $L_0\pitchfork H$ in $G$. Then the following conditions are equivalent:
 \begin{eqenum}
  \item The $L$-action on $G/N$ and the $L_0$-action on $N/H$ are cocompact.
  \item The $L$-action on $G/H$ is cocompact.
 \end{eqenum}
\end{prop}

\begin{proof}
First, we show the implication (a)$\imply$(b). By the assumption $L_0\pitchfork H$ in $G$ and the cocompactness of the $L_0$-action, for any $g\in G$, we have $\A_g(L_0)\pitchfork H$ in $N$ and $\Ad_{g}\fl_0\oplus\fh=\fn$, and so $\A_{g}(L_0)H=N$ (Proposition~\ref{Prop:solvablecri}). Take a compact subset $C\subset G$ satisfying $G=LCN$. We are enough to show $G=LCH$. Then we get:
	\[
	 G=L CN=\bigcup_{c\in C}LcN=\bigcup_{c\in C}Lc\A_{c\inv}(L_0)H=\bigcup_{c\in C}LL_0cH=LCH.
	\]
Next, we show the implication (b)$\imply$(a). By $H\subset N$, the $L$-action on $G/N$ is cocompact, so we show that the $L_0$-action is compact. Take a compact subset $C\subset G$ satisfying $G=LCH$. Then we have:
	\[
	 N=L_0L_1 CH\cap N=L_0(L_1C\cap N)H.
	\]
	By the condition $L_1\pitchfork N$, we have $L_1C\cap N$ is compact, and so have the $L_0$-action on $N/H$ is cocompact.
\end{proof}

\section{Indecomposable symmetric triples with signature $(2,2)$}\label{Section:Indecomposablesymmetrictriples}

In this section, we review the classification of symmetric triples with signature $(2,2)$ given by Kath--Olbrich (Fact~\ref{Fact:classification}). We use another notation for our calculations. \par
We introduce solvable Lie algebras $\fg_{D,D'}$ and indecomposable symmetric triples $\ft_{D,D'}:=(\fg_{D,D'},\sigma,g)$, which are most part of the pseudo-Riemannian symmetric triple with signature $(2,2)$ (see Fact~\ref{Fact:classification}). We also define a symmetric triples $(\fg_{\rm nil},\sigma,g_{\pm})$ in Subsection~4.3.\par
In this section, we use the following:
\begin{notation}\label{notation:4}
\begin{itemize}
  \item $\fh_n:=\R^{2n}\oplus \R $ : the $(2n+1)$-dimensional Heisenberg Lie algebra equipped with the following non-trivial brackets:
	\[
	 [X,Y]:=\omega(X,Y),\quad X,Y\in\R^{2n}
	\]
	where $\omega$ is a symplectic form on $\R^{2n}$. Put $\fh_0:=\R$.	
  \item $\fz:=\R$ : the center of $\fh_n$, $Z:=1\in\fz$.
  \item $H_n$ : Heisenberg Lie group, namely, 1-connected Lie group whose Lie algebra is $\fh_n$.
\end{itemize} 
\end{notation}

\begin{rem}\label{Rem:correspondance}
The correspondence between the spaces of Fact~\ref{Fact:classification} and the list of Kath--Olbrich \cite[Theorem~7.1]{OK} is given by the following table:\\
 \begin{table}[h]
 \begin{center}
  \begin{tabular}{c|c|c|c|c|c}
    \cite[Theorem~7.1]{OK} & (1) (a)(b)& (1)(c) & (2)(a)(b) &(3)&(4)\\\hline
    Fact~\ref{Fact:classification} & II(a),(d) & II(b) & I &II(c) $D'=D'_\ep$ & II(c) $D'=-D'_\ep$
  \end{tabular},
 \end{center}
 where $D'_\ep :=  \begin{pmatrix}0&1\\ 1&\ep\end{pmatrix}$ and $\ep=\pm1$.
 \end{table}
\end{rem}

\subsection{Definition of symmetric triples $\ft_{D,D'}$}\label{gddexample}

In this subsection, we define symmetric triples $\ft_{D,D'}$ and see some properties.

\begin{defi}\label{Def:gw}
 We think $\fs\fp(n,\R)$ as a subalgebra of $\Der \fh_n$. For $W\in\fs\fp(n,\R)$, we define:
 \[
  \fg_{W}:= \R W\ltimes \fh_n\subset \fs\fp(n,\R)\ltimes \fh_n.
 \]
\end{defi}

In the following, we use a basis $(X_1,\cdots,X_n,Y_1,\cdots,Y_n)$ of $\R^{2n}$ such that: 
\[
 \omega=
 \begin{pmatrix}
  &I_n\\ -I_n&
 \end{pmatrix}.
\]

Using this basis, we identify $\fs\fp(n,\R)\simeq \set{W\in M(2n,\R)}{\omega W+W^T\omega =0}$.

\begin{defi}[$\fg_{D,D'}$]\label{Def:gdd}
 For matrices $D,D'\in\sym^{(\text{reg})} (n,\R)$, we put $W:=\begin{pmatrix}&D'\\D&\end{pmatrix}\in\asp(n,\R)$. We define a subalgebra $\fg_{D,D'}\subset \fs\fp(n,\R)\ltimes\fh_{n}$ by $\fg_{D,D'}:=\fg_{W}$.
\end{defi}

\begin{lem}\label{Lem:eigenvalues}
 For $D,D'\in \sym^{(\text{reg})}(n,\R)$, we have the following statements.
 \begin{enumerate}
  \item The Lie algebra $\fg_{D,D'}$ is solvable and $\dim \fg_{D,D'}=2n+2$.
  \item The eigenvalues of $W$ are square roots of the eigenvalues of the product $DD'$.	
  \item The Lie algebra $\fg_{D,D'}$ is completely solvable if and only if all the eigenvalues of the product $DD'$ are positive real numbers.
 \end{enumerate}
\end{lem}

\begin{proof}
 Since the statements (1) and (2) are clear, we show the statement (3). In general, a Lie algebra over $\R$ is completely solvable if and only if all the eigenvalues of the adjoint representations are real. In our case, since $\fh_n$ is a nilpotent ideal, it is equivalent to the eigenvalues of $\ad (W+U)$ on $\fg_{D,D'}$ are real for any $U\in\fh_n$. By an easy calculation, with respect to the base $(W,X_1,\cdots,X_n, Y_1,\cdots, Y_n,Z)$, we have:
	\[\ad (W+U)=\begin{pmatrix}
		  0&&&\\
	    *&&D'&\\
	   *&D&&\\
	0&*&*&0	 
		     \end{pmatrix}.\]
	Hence, $\fg_{D,D'}$ is completely solvable if and only if all the eigenvalues of $W=\begin{pmatrix}0&D'\\ D&0  \end{pmatrix}$ are real. By the statement (2), it is equivalent to the condition that all the eigenvalues of the product $DD'$ are positive.
\end{proof}

\begin{propdefi}\label{propanddefi:inpro}
For $D,D'\in\sym^{(\text{reg})} (n,\R)$, set a subalgebra $\fh\coloneqq\xspan{\R}{Y_1,\cdots,Y_n}$ and a subspace $\fq\coloneqq \R W\oplus \xspan{\R}{X_1,\cdots,X_n}\oplus \fz\subset \fg_{D,D'}$. Let $g$ be the inner product on $\fq\subset \fg_{D,D'}$ defined by the following Gram matrix with respect to the basis $(W, X_1,\cdots,X_n, Z)$:
 	\[
	 g= \begin{pmatrix}
	    0&&-1\\
	    &D'^{\inv}&\\
	    -1&&0
	   \end{pmatrix}.
	\]
 Then we have:
 \begin{enumerate}
  \item $\fg_{D,D'}=\fq\oplus\fh$,
  \item $[\fq,\fh]\subset \fq$ and $[\fq,\fq]= \fh$,
  \item $g$ is $\fh$-invariant.
 \end{enumerate}
 Especially, by Note~\ref{Note:kousei}, we construct a symmetric triple $\ft_{D,D'}:=(\fg_{D,D'},\sigma, g)$ with signature $(p+1, q+1)$, where $(p,q)$ is the signature of $D'$.
\end{propdefi}

\begin{proof}
 The statements (1) and (2) $[\fq,\fh]\subset \fq$ are clear. Since $D$ is invertible, we have $[\fq,\fq]= \fh$.
 Then we show the statement (3), namely,
		\[
	 g((\ad v) X, Y)+g(X,(\ad v) Y)=0\quad (\forall X,Y\in\fq, \forall v\in\fh).
	\]
Regard $g$ as the representation matrix of the inner product on $\fq=\R W\oplus \xspan{\R}{X_1,\cdots,X_n} \oplus \fz$ and let $A\in M(2n+2,\R)$ denote the representation matrix of the linear translation $\ad v\in\End(\fq)$. Then this condition is equivalent to the condition $gA+(gA)^T=0$, namely, $gA$ is skew symmetric. By an easy calculation, we have:
\[
 A=\begin{pmatrix}
	  &&\\
	  -D'v &&\\
	  &-v^T&\quad\\
	 \end{pmatrix},\quad
gA=\begin{pmatrix}
	  &v^T&\\
	  -v&&\\
	  &&
	  \end{pmatrix}.
\]
\end{proof}

We denote by $G_{D,D'}$ the 1-connected Lie group with the Lie algebra $\fg_{D,D'}$ and by $H\subset G_{D,D'}$ the analytic subgroup with respect to $\fh$. By Fact~\ref{Fact:transvection}, $G_{D,D'}$ is the transvection group of $G_{D,D'}/H$.

\begin{prop}\label{Prop:isomofgdd}
 For any $D, D'\in\sym^{(\text{reg})}(n,\R)$, the symmetric triple $\ft_{D,D'}$ is reducible and indecomposable.
\end{prop}

\begin{proof}
 Since the subspace $\fz\subset \fq$ is $\ad \fh$-invariant, the symmetric triple $\ft_{D,D'}$ is reducible. Then we show the indecomposability.
 Let $\fg_{D,D'}=\fg_1\oplus\fg_2$ be a non-trivial decomposition. Then we have a $\fh$-invariant decomposition $\fq=\fq_1\oplus\fq_2$. By $[\fq_i,\fq_i]=\fh_i$ for $i=1,2$, the subspace $\fq_1$ is non-trivial. Then we are enough to show that $\fq_1$ is degenerate. By Note~\ref{Note:gperp} below, it is enough to show that $\fz\subset \fq_1 \subset \mathcal{X}\oplus\fz$, where $\mathcal{X}:=\xspan{\R}{X_1,\cdots, X_n}$. For the decomposition $\fq= \R W \oplus\mathcal{X}\oplus \fz$, we set the projections $\pr_1:\fq\to \R W,\ \pr_2:\fq\to \mathcal{X}$ and $\pr_3:\fq \to \fz$, respectively.  
First, we show $\fq_1\subset \mathcal{X}\oplus \fz$. Let $v\in \fq_1$ and assume $\pr_1(v)\neq 0$. Since $D$ is invertible, we obtain $\mathcal{X}\subset (\ad\fh) v+\fz$ and $\fz\subset (\ad \fh)^2 v$, and so we have $\fq_1=\fq$, which contradicts the non-triviality of $\fq_1$. Therefore, $\pr_1(v)=0$ for any $v\in \fq_1$ and so we have $\fq_1\subset \mathcal{X}\oplus\fz$. Next, we show $\fz\subset \fq_1$. If $\pr_2(v)=0$ for any $v\in \fq_1$, we have $\fz\subset \fq_1$, so we take $v\in \fq_1$ satisfying $\pr_2(v)\neq 0$. Then we have $\fz\subset  (\ad \fh) v$, and so $\fz\subset \fq_1$.
\end{proof}

\begin{note}\label{Note:gperp}
 Put $\fz^\perp:=\set{x\in\fg_{D,D'}}{g(x,\fz)=\{0\}}$, then we have $\fz^\perp = \xspan{\R}{X_1,\cdots, X_n}\oplus \fz$.
\end{note}

\subsection{Isomorphic classes of the triples $\ft_{D,D'}$}

In this subsection, we see an isomorphic classes of the triple $\ft_{D,D'}$ (Proposition~\ref{Prop:isomclass}) and give the classification for triples with signature $(2,2)$ (Proposition~\ref{Prop:bunruiteiri}).

\begin{prop}\label{Prop:isomclass}
 For $D_1,D_2,D_1',D_2'\in\sym^{(\text{reg})} (n,\R)$, two symmetric triples $\ft_{D_1,D'_1}$ and $\ft_{D_2,D'_2}$ are isomorphic if and only if there exists $(P,k)\in GL(n,\R)\times\R_{>0}$ satisfying:
	\[
	PD_1'P^T=D_2'\ {\rm and}\ k P^TD_2P= D_1.
	\]
\end{prop}
 
\begin{proof}
  Take an isometric Lie algebra homomorphism $\phi:\fg_{D_1,D_1'}\to\fg_{D_2,D_2'}$, which is compatible with the involutions. Then $\phi$ preserves:
	\begin{itemize}
	 \item the decomposition $\fh\oplus\fq$,
	 \item the center $\fz$,
	 \item its orthogonal subspace $\fz^\perp$.
	\end{itemize}
	Therefore, with a basis $(W,X_1,\cdots,X_n,Y_1,\cdots, Y_n, Z)$, the map $\phi$ is written of the form:
	\begin{align*}
	\phi=\begin{pmatrix}
	       \ell_1&&&\\
	       *&P&&\\
	       &&Q&\\
	       *&*&&\ell_2
	      \end{pmatrix},	 
	\end{align*}
	 where $\ell_1, \ell_2\in \R^\times$ and $P, Q\in GL(n,\R)$. Since $\phi$ preserves the inner products, we have:
	\begin{align*}
	 g(\phi(X_i),\phi(X_j))&=g(X_i,X_j)\quad (i,j=1,2,\cdots,n)&&\iff P^T{D_2'}\inv P={D_1'}\inv ,\\
	 g(\phi(W),\phi(Z))&=g(W,Z)&&\iff -\ell_1\ell_2=-1.
	\end{align*}
	Since $\phi$ is a Lie algebra homomorphism, we have
	\begin{align*}
	 \phi([X_i,Y_j])&=[\phi (X_i),\phi (Y_j)] \quad (i,j=1,2,\cdots,n) & &\iff \ell_2 I_n=PQ^T,\\
	 \phi([W,X_i])&=[\phi (W),\phi (X_i)] \quad (i=1,2,\cdots,n)  &&\iff QD_1=\ell_1D_2P.
	\end{align*}
	Set $k:=\ell_1^2$, then the conditions $PD_1' P^T=D_2'$ and $k P^TD_2P=D_1$ follow from the above calculations. Conversely, if there exists $(P,k)\in GL(n,\R)\times \R_{>0}$ satisfying the condition, a direct calculation leads us that the homomorphism $\phi$ obtained by putting $*=0$ in the matrix representation above is an isomorphism of symmetric triples.
\end{proof}

\begin{defi}\label{Def:equiv}
 We denote by $(D_1, D'_1) \sim (D_2,D'_2)$ the condition of Proposition~\ref{Prop:isomclass}.
 It is easy to check that $\sim$ is an equivalence relation on $\sym^{(\text{reg})} (n,\R) \times \sym^{(\text{reg})} (n,\R)$.
\end{defi}

 In the rest of this subsection, we see the isomorphic classes of the triples $\ft_{D,D'}$ with signature $(2,2)$ with respect to this equivalence relation.

\begin{prop}\label{Prop:bunruiteiri}
 For matrices $D,D'\in \sym^{(\text{reg})}(2,\R)$, assume the signature of $D'$ is $(1,1)$. Then the following list gives a complete class representatives of symmetric triple $(\fg_{D,D'},\sigma,g)$.
  \begin{enumerate}
  \item $(D,D')=(\pm\diag{1,\nu},\diag{1,-\nu})\quad (\nu>0)$,\\ $(D,D')=(\pm\diag{1,-\nu},\diag{1,-\nu})\quad (\nu>0, \ \nu\neq 1)$, 
  \item $(D,D')=\left(Q_\nu,Q_{-\nu}\right)\quad (\nu>0)$\ (see Notation~\ref{Not:basic}),
  \item $(D,D')=\left(\begin{pmatrix}
	    \pm1&-1\\
	    -1&0
	   \end{pmatrix},
	\begin{pmatrix}
	 0& -1\\ -1&\pm 1
	\end{pmatrix}\right),\
	\left(\begin{pmatrix}
	    \pm1&-1\\
	    -1&0
	      \end{pmatrix},
	\begin{pmatrix}
	 0&1\\ 1&\mp1
	\end{pmatrix}\right)$,
   \item $(D,D')=(\pm I_{1,1},I_{1,1})$.
  \end{enumerate}
\end{prop}

\begin{proof}[\bf{Proof of Proposition~\ref{Prop:bunruiteiri}}]
 By Proposition~\ref{Prop:isomclass},\ we may and do assume $D'= I_{1,1}$. For a basis $(I_2,Q_0,I_{1,1})$ of $\sym(2,\R)$, we put $D(x,y,z):=xI_2+y Q_0+zI_{1,1}\in\sym(2,\R)$. Note that $\det D(x,y,z)=x^2-y^2-z^2$.
Then we are enough to consider the orbit of $D(x,y,z)$ with respect to the action $O(1,1)\times \R_{>0}\to \isom(\sym^{(\text{reg})}(2,\R)),\ (P,k) \mapsto (D\mapsto kPDP^T)$.
 Note that:
 \begin{align*}
  O(1,1)&=\Big\langle H(t):=\begin{pmatrix}
	       \cosh t& \sinh t\\
	       \sinh t& \cosh t
			  \end{pmatrix},\ I_{1,1},\ \pm I_2\Big\rangle_{t\in\R},\\
  H(t) D(x,y,z) H(t)^T&=D(x\cosh 2t+y\sinh 2t, x\sinh 2t+y\cosh 2t, z),\\
  I_{1,1}D(x,y,z)I_{1,1}^T&=D(x,-y,z).
 \end{align*}
 Set $G:=\gen{
  \diag{H(2t),1},
  \diag{1,\pm1,1},
  k I_3}_{t\in\R, k\in\R_{>0}}$. 
 Then we are enough to consider the orbit space of the $G$-action on $\set{(x,y,z)\in\R^3}{x^2-y^2-z^2\neq 0}\subset \R^3-\{0\}$.
 By an easy calculation, we have: 
 \begin{note}
The orbit space of the $G$-action on $\R^3-\{0\}$ is:
  \[
   \{[(\pm1,0,z)],[(0,1,z)],[(1,1,\pm 1)],[(-1,1,\pm 1)],[(0,0,\pm1)]\}_{z\in\R}.
  \]
 \end{note}
 \begin{enumerate}
 \item The orbits $[(\pm1,0,z)]$ ($z\in\R,\ z\neq \pm1$).\\
       Here, the constraint $z\neq \pm1$ comes from the condition $x^2-y^2-z^2\neq 0$. We have $(D,D')\sim (\diag{1+z,1-z},I_{1,1}),\ (\diag{-1+z,-1-z},I_{1,1})$ and:
       \begin{align*}
	(\diag{1+z,1-z},I_{1,1})&\sim \begin{cases}(\diag{1,\mu},I_{1,1})\quad (1+z>0,\mu>-1, \mu\neq 0)\\ (\diag{-1,\mu},I_{1,1}) \quad (1+z<0,\mu>1)\end{cases},\\
	(\diag{-1+z,-1-z},I_{1,1})&\sim \begin{cases}(\diag{1,\mu},I_{1,1})\quad (-1+z>0,\mu<-1)\\ (\diag{-1,\mu},I_{1,1}) \quad (-1+z<0,\mu<1, \mu\neq 0)\end{cases}.
       \end{align*}       
       Then we get $(D,D')\sim (\pm \diag{1,\mu},I_{1,1})$ for some $\mu\in\R^\times$ ($\mu\neq-1$). Putting $\nu := \sqrt{|\mu|}$ and $P:=\diag{1,\sqrt{\nu}}$, we obtain $(\diag{\pm1,\mu},I_{1,1})\sim (\pm\diag{1,\nu},\diag{1,-\nu})$ or $(\pm\diag{1,-\nu},\diag{1,-\nu}), \nu\neq 1$.
  \item The orbits $[(0,1,z)]$ ($z\in\R$).\\
	In this case, we have: 
	\[
	 (D,D')\sim \left(\begin{pmatrix}
			   z&1\\ 1&-z
			 \end{pmatrix},I_{1,1}\right )\sim \left(\begin{pmatrix}
			   1-\nu^2&\nu\\ \nu &-1+\nu^2
			 \end{pmatrix},I_{1,1}\right )
	\]
	for some $\nu\in\R_{>0}$. Take $t\in\R$ satisfying $\sinh t=\nu$ and put:
	\[
	 P:=\dfrac{e^{-t/2}}{\sqrt{2}}\begin{pmatrix}
					e^{t}&1\\1&-e^{t}
				       \end{pmatrix},
	\]
	then we have the equivalence $(D,D')\sim \left(Q_\nu, Q_{-\nu} \right)$. 
  \item The orbits $[(1,1, \pm1)],[(-1,1,\pm1)]$.\\
	In this case, we have:
	\begin{align*} (D,D')\sim
	 \left(\begin{pmatrix}
	  0&1\\ 1&2
	 \end{pmatrix},I_{1,1}\right),
	 \left(\begin{pmatrix}
	  -2&1\\ 1&0
	 \end{pmatrix},I_{1,1}\right),
	 \left(\begin{pmatrix}
	  2&1\\ 1&0
	 \end{pmatrix},I_{1,1}\right),
	 \left(\begin{pmatrix}
	  0&1\\ 1&-2
	 \end{pmatrix},I_{1,1}\right).
	\end{align*}
	By a direct calculation, we have the equivalence to the following classes, respectively. (For example, we use $P:=\dfrac{1}{2\sqrt{2}}\begin{pmatrix} 2&2\\ 1&-3\end{pmatrix}$ for the first pair.)
	\begin{align*}
	 \left(\begin{pmatrix}
		1&-1\\-1&0
	       \end{pmatrix},
	 \begin{pmatrix}
	  0&1\\ 1&-1
	 \end{pmatrix}\right), 
	 \left(\begin{pmatrix}
		-1&-1\\-1&0
	       \end{pmatrix},
	 \begin{pmatrix}
	  0&1\\ 1&1
	 \end{pmatrix}\right),\\ 
	 \left(\begin{pmatrix}
		1&-1\\-1&0
	       \end{pmatrix},
	 \begin{pmatrix}
	  0&-1\\ -1&1
	 \end{pmatrix}\right), 
	 \left(\begin{pmatrix}
		-1&-1\\-1&0
	       \end{pmatrix},
	 \begin{pmatrix}
	  0&-1\\ -1&-1
	 \end{pmatrix}\right).
	\end{align*}	
  \item The orbits $[(0,0,\pm1)]$.\\
	In this case, we have $(D,D')\sim\left(\pm I_{1,1},I_{1,1}\right)$.
 \end{enumerate}
\end{proof}

\begin{rem}
 With the natural identification $(\R^3-\{0\})/\R_{>0}\simeq \mathbb{S}^2$ (the 2-dimensional unit sphere), a picture of the parameter spaces of $G_{D,I_{1,1}}$ is given as Figure~\ref{fig1}.
 \begin{figure}[h]
  \begin{center}
  \includegraphics[width = 6cm]{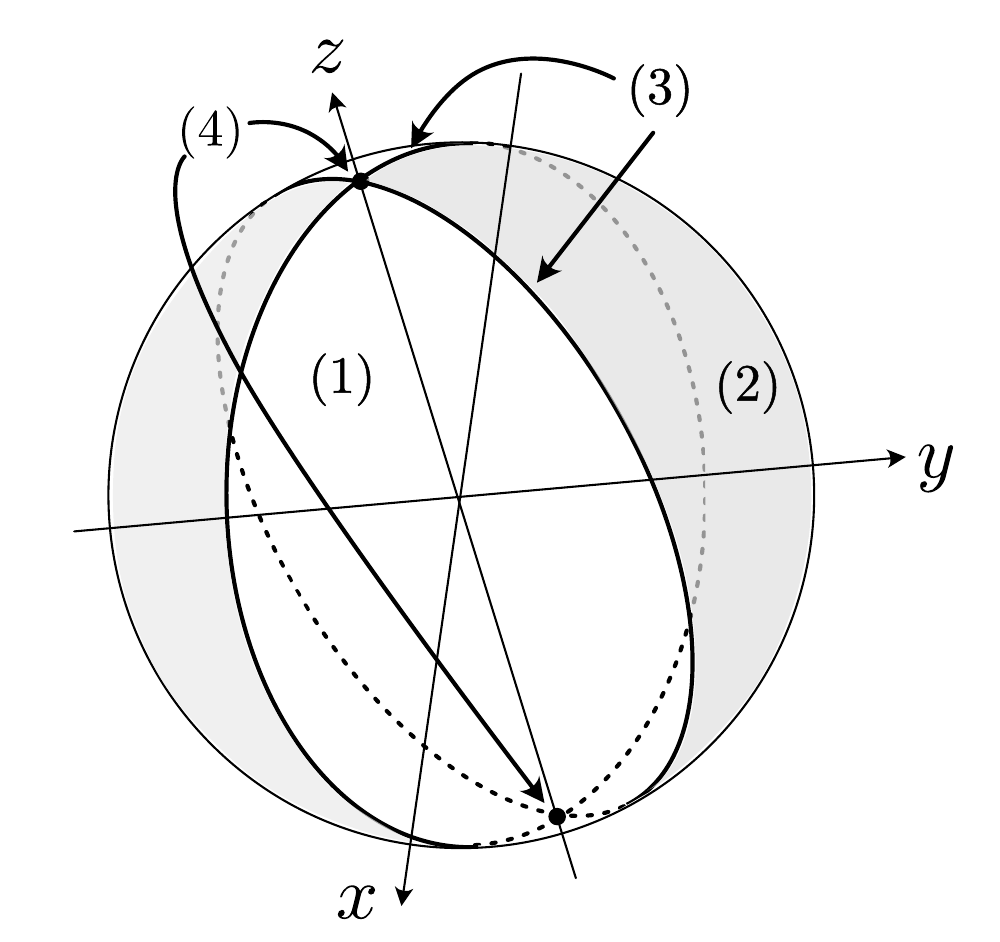}
   \caption{A picture of the parameter space of $G_{D,I_{1,1}}$}   
  \label{fig1}
  \end{center}
 \end{figure}
\end{rem}

\begin{rem}
For a symmetric triple $(\fg_{D,D'},\sigma,g)$, put $\alpha,\beta\in\C$ eigenvalues of the product $DD'$. Then the following table shows which class of Proposition~\ref{Prop:bunruiteiri} the symmetric triple $(\fg_{D,D'},\sigma,g)$ belongs to.
  \begin{table}[h]
 \begin{center}
  \begin{tabular}{c|c|c|c|c}
    class & (1) & (2) & (3) & (4)\\\hline
eigenvalues & real & not real & real &real\\\hline
  relation & $\alpha\neq \beta$ & $\alpha\neq \beta$ & $\alpha=\beta$&$\alpha=\beta$\\\hline
    $DD'$ is diagonalizable &yes&yes&no&yes
  \end{tabular}
 \end{center}
 \end{table}
\end{rem}

\subsection{Definition of symmetric triples $(\fg_{\rm nil}, \sigma, g_{\pm})$}


In this subsection, we define a symmetric triples $(\fg_{\rm nil},\sigma, g_\pm)$, which is in the list of the classification (Fact~\ref{Fact:classification}).

\begin{defi}[$\fg_{\text{nil}}, G_{\text{nil}}/H$]\label{Def:gnil}
We define a nilpotent Lie algebra\\ $\fg_{\rm nil}\coloneqq\xspan{\R}{A_1,A_2,B,C_1,C_2}$ as follows:
 \begin{align*}
  &[A_1,A_2]=B,\quad
  [B,A_1]=C_1, \quad
  [B,A_2]=C_2,\\
  &\text {the other brackets are trivial.}
 \end{align*}
Put $\fh\coloneqq\R{B}$ and $\fq\coloneqq\xspan{\R}{A_1, A_2, C_1, C_2}$, then we have $\fg_\text{nil}=\fq\oplus\fh$, $[\fq,\fh]\subset \fq$ and $[\fq,\fq]=\fh$. We define a $\fh$-invariant inner product $g$ on $\fq$ as follows:
 \[
 g_\pm\coloneqq \begin{pmatrix}
     &\pm J\\\\
     \pm J&
    \end{pmatrix},
 \]
 where $J:=\begin{pmatrix}
	    &-1\\ 1&
	   \end{pmatrix}$.
By Note~\ref{Note:kousei}, the triples $(\fg_{\text{nil}},\sigma, g_{\pm})$ are indecomposable symmetric triples with signature $(2,2)$. 
We denote by $G_{\rm nil}$ the 1-connected nilpotent Lie group with the Lie algebra $\fg_{\rm nil}$ and by $H$ the analytic subgroup of $G_{\rm nil}$ with respect to $\fh$. The Lie group $G_{\rm nil}$ is the transvection group of $G_{\rm nil}/H$ (Fact~\ref{Fact:transvection}).
\end{defi}




\section{Criterions of the existence of compact Clifford--Klein forms for spaces $G_{D,D'}/H$}\label{Section:Criterions}

To prove the main theorem, we prepare some criterions for the existence of compact Clifford--Klein forms of the symmetric space $G_{D,D'}/H$ (Propositions~\ref{Prop:mainprop} and \ref{Prop:csmainprop}). In this section, we use Notation~\ref{notation:4} and the following:
\begin{notation}\label{notation:second} 
 \begin{itemize}
  \item $\fh:=\xspan{\R}{Y_1,\cdots, Y_n}\subset \fh_n$,
  \item $D, D'\in\sym^{(\text{reg})}(n,\R)$,
  \item $W\coloneqq\begin{pmatrix}&D'\\D&\end{pmatrix}\in M(2n,\R),\quad W_t:=\exp tW=\begin{pmatrix}A_t&B_t\\ *&*\end{pmatrix}\in GL(2n,\R)$,
  \item $\pr_1:G_{D,D'}=\R\ltimes H_n\to \R \text{ the first projection}$,
  \item $\pr_2:G_{D,D'}=\R\ltimes H_n\to H_n \text{ the second projection}$.
 \end{itemize}
\end{notation}

In this section, we identify $\xspan{\R}{X_1,\cdots, X_n, Y_1,\cdots, Y_n}$ as $\R^{2n}$ and $\fh_n$ as $\R^{2n+1}$. Especially, we think $M(2n,\R)\simeq \operatorname{End}({\xspan{\R}{X_1,\cdots,X_n, Y_1\cdots, Y_n}})$ and $M(2n+1,\R)\simeq \operatorname{End}({\xspan{\R}{X_1,\cdots,X_n, Y_1\cdots, Y_n, Z}})$.

\begin{note}\label{Note:easycalc}
 By a direct calculation, we have:
 \[
  \Ad_g|_{\fh_n} = \begin{pmatrix} W_{\pr_1(g)}&\\ * &1\end{pmatrix},\quad (g\in G_{D,D'}).
 \]
\end{note}

\subsection{Subgroups $L_C$ and $L_{C,w}$}

In Subsection~\ref{Subsection:Criterions}, we see criterions for the existence of compact Clifford--Klein forms. Before that, we show two basic Propositions \ref{Prop:properequiv} and \ref{Prop:structurethm}. In Proposition \ref{Prop:structurethm}, we classify constructors of $G_{D,D'}/H$. To do this, we introduce some subgroups of $G_{D,D'}$ and show their basic properties.

\begin{defi}
 For $C\in M(n,\R)$ and $w\in\fh$, we denote by $\fl'_{C}\subset \R^{2n}$ the image of the linear transform defined by
 $\begin{pmatrix}
   I_n&0\\
   C&0
  \end{pmatrix}$.
Then we put:
 \begin{align*}
 \fl_{C}&:=\fl_{C}'\oplus\fz\underset{\rm subalgebra}\subset \fh_{n},\\
 \fl_{C,w}&:=\R(W+w)\oplus\fl_{C}\underset{\rm subspace}\subset \fg_{D,D'}.
 \end{align*}
We denote by $L_C$ the analytic subgroup in $G_{D,D'}$ with respect to $\fl_C$. If $\fl_{C,w}$ is a subalgebra of $\fg_{D,D'}$, we denote by $L_{C,w}$ its analytic subgroup in $G_{D,D'}$. 
\end{defi}

We see the criterion of $\fl_{C,w}$ to be a subalgebra of $\fg_{D,D'}$.

\begin{prop}\label{Prop:subalgebracri}
 For $C\in M(n,\R)$ and $w\in\fh$, the following conditions are equivalent:
 \begin{eqenum}
  \item the subspace $\fl_{C,w}\subset \fg_{D,D'}$ is a subalgebra,
  \item $[W+w,\fl_C]\subset \fl_C$,
  \item the subspace $\fl'_C$ is $W$-invariant,
  \item the subspace $\fl'_C$ is $W_t$-invariant\quad  ($\forall t\in\R$),
  \item the subalgebra $\fl_C$ is $\Ad_{G_{D,D'}}$-invariant,
  \item $CD'C=D$.
 \end{eqenum}
\end{prop}

\begin{rem}
 By this proposition, the conditions (a) and (b) do not depend on $w\in\fh$.
\end{rem}

\begin{proof}
 Since $\fl_C$ is a subalgebra of $\fh_n$, the condition (a) is equivalent to the condition (b).
 By $[w,\fl_C]\subset\fz$, we have the equivalence (b)$\siff$(c) by a direct calculation. The equivalence (c)$\siff$(d) is easy and we have the implication (d)$\siff$(e) by Note \ref{Note:easycalc}. By Note~\ref{Note:linearlem} below, we show the equivalence (c)$\siff$(f) as follows.
 \begin{align*}
  \text{(c)} \iff (C, -I_n)\begin{pmatrix}
			&D'\\D&
		       \end{pmatrix}
  \begin{pmatrix}
   I_n\\ C
  \end{pmatrix}=0 \iff CD'C-D=0 \iff \text{(f)}.
 \end{align*} 
\end{proof}

\begin{note}\label{Note:linearlem}
 For $C\in M(n,\R)$ and $A\in M(2n,\R)$, the subspace
 $\im \begin{pmatrix}
       I_n\\ C
      \end{pmatrix}\subset\R^{2n}$ is $A$-invariant if and only if
 $(C,\ -I_n)A
 \begin{pmatrix}
  I_n\\ C
 \end{pmatrix}=O$.
\end{note}

We give fundamental properties of the subspaces $\fl_{C}$ and $\fl_{C,w}$.

\begin{note}\label{Note:structurelcw}
  There is a Lie algebra isomorphism $\fl_{C} \simeq\fh_{k}\oplus \R^{n-2k}$, where $k\coloneqq(\operatorname{rank}(C-C^T))/2$. 
\end{note}

\begin{note}\label{Note:lcwdecomposition}
 For $C\in M(n,\R)$ and $w\in\fh$, we have decompositions $\R^{2n}=\fl_C'\oplus\fh$, $\fh_n=\fl_C\oplus\fh$ and $\fg_{D,D'}=\fl_{C,w}\oplus\fh$ as linear spaces.
\end{note}

The ``converse of Note~\ref{Note:lcwdecomposition}'' also holds. In fact, we have:

\begin{prop}\label{Prop:decomposition}
 \begin{enumerate}
  \item For any subspace $\fl\subset \R^{2n}$ satisfying $\R^{2n}=\fl\oplus \fh$, there exists $C\in M(n,\R)$ satisfying $\fl=\fl'_C$.
  \item For any subalgebra $\fl\subset\fh_{n}$ satisfying $\fh_{n}=\fl\oplus\fh$, there exists $C\in M(n,\R)$ satisfying $\fl=\fl_C$.
  \item  For any subalgebra $\fl\subset\fg_{D,D'}$ satisfying $\fg_{D,D'}=\fl\oplus\fh$, there exist $C\in M(n,\R)$ and $w\in \fh$ satisfying $CD'C=D$ and $\fl=\fl_{C,w}$.
 \end{enumerate}
\end{prop}

To prove this proposition, we use the following fact and note.

\begin{fact}[{\cite[Lemma~4.3]{OKCK}}]\label{Fact:centerin}
If a subalgebra $\fl\subset \fh_n$ satisfies $\fh_n=\fl\oplus \fh$ as a linear space, we have $\fz\subset \fl$.
\end{fact}

\begin{note}\label{Note:subsetnote}
 Let $V=U\oplus W$ be a linear space decomposition. For a subspace $V_1\subset V$ satisfying $W\subset V_1$, we have $V_1=(V_1\cap U)\oplus W$.
\end{note}

\begin{proof}[{\rm \bf Proof of Proposition~\ref{Prop:decomposition}}]
 \begin{enumerate}
  \item Let $\fl\subset \R^{2n}$ be a subspace satisfying $\R^{2n}=\fl\oplus\fh$. Then there exists $C\in M(n,\R)$ such that:
	\[
	 \fl=\im
	\begin{pmatrix} I&0\\ C&0\end{pmatrix}=\fl'_C.
	\]
  \item Let $\fl\subset \fh_n$ be a subalgebra satisfying $\fh_n=\fl\oplus\fh$. By using Note~\ref{Note:subsetnote} for the decomposition $\fh_n = \fl\oplus\fh$ and the subspace $\R^{2n}\subset \fh_n$, we have $\R^{2n}=(\R^{2n}\cap \fl)\oplus\fh$. By the statement (1), there exists $C\in M(n,\R)$ satisfying $\R^{2n}\cap \fl=\fl'_C$. By Fact~\ref{Fact:centerin}, we get $\fz\subset \fl$. Then by using Note~\ref{Note:subsetnote} again for the decomposition $\fh_n=\R^{2n}\oplus\fz$ and the subspace $\fl\subset\fh_n$, we have $\fl=(\fl\cap \R^{2n})\oplus \fz=\fl'_C\oplus\fz=\fl_C$.
  \item Let $\fl\subset \fg_{D,D'}$ be a subalgebra satisfying $\fg_{D,D'}=\fl\oplus\fh$. By using Note~\ref{Note:subsetnote} for the decomposition $\fg_{D,D'}=\fl\oplus\fh$ and the subspace $\fh_n\subset \fg_{D,D'}$, we get $\fh_n=(\fh_n\cap\fl)\oplus\fh$. By the statement (2), we have $\fh_n\cap \fl=\fl_C$ for some $C\in M(n,\R)$. Then we have $\fl=\R(W+w)\oplus \fl_C=\fl_{C,w}$ for some $w\in \fh$. Since $\fl_{C,w}$ is subalgebra, we obtain $CD'C=D$ (Proposition~\ref{Prop:subalgebracri}).
 \end{enumerate}
\end{proof}

\begin{prop}\label{Prop:properequiv}
For $C\in M(n,\R)$, the following conditions are equivalent:
 \begin{eqenum}
	\item the pair $(L_C, H)$ satisfies the property (CI) in $G_{D,D'}$,
        \item $\Ad_{G_{D,D'}}\fl_C\cap\fh=\{0\}$,
        \item $W_t \fl'_C\cap \fh=\{0\}$ \quad ($\forall t\in\R$),
	\item the matrix $A_t + B_t C$ is invertible \quad ($\forall t\in \R$).
 \end{eqenum}
\end{prop}
\begin{proof}
The equivalence (a)$\siff$(b) comes from Note~\ref{Note:nilpci2}.
We have $\Ad_{g}\fl_C=W_{\pr_1(g)}\fl'_C\oplus\fz$ for $g\in G_{D,D'}$ by Note~\ref{Note:easycalc}, so the equivalence (b)$\siff$(c) holds.\par
 Then we show the equivalence (c)$\siff$(d). Take any $t\in\R$. Since $W_t\fl_C'$ is the image of the linear map $\begin{pmatrix}A_t&B_t\\ *&*\end{pmatrix}\begin{pmatrix}I_n&0\\ C&0\end{pmatrix}$, we have:
 \[
 W_t\fl'_C\cap \fh=\{0\}\siff
  \det(A_t+B_tC)\neq0.
 \]
\end{proof}

Then we get:

\begin{lem}\label{Lem:LCCI}
 Suppose $C\in M(n,\R)$ satisfies $CD'C=D$, then we have $L_{C,w}\pitchfork H$ in $G_{D,D'}$. 
\end{lem}

\begin{proof}
By Note~\ref{Note:nilpci2} and Propositions~\ref{Prop:subalgebracri} and \ref{Prop:properequiv}, the pair $(L_C,H)$ has the property (CI) in $H_n$. By Fact~\ref{Fact:nasrin}, we have $L_C\pitchfork H$ in $H_n$. By the condition $\A_g L_C=L_C$ for any $g\in G_{D,D'}$, we have $\A_{S}L_C\pitchfork H$ in $H_n$ for any compact set $S\subset G_{D,D'}$. By using Lemma~\ref{Lem:properlem} with $(G,N,L)=(G_{D,D'},H_n,L_1L_C)$, we obtain $L_{C,w}\pitchfork H$ in $G$, where $L_1$ is the analytic subgroup of $G_{D,D'}$ with respect to $\R(W+w)$ and the condition $L_1\pitchfork H_n$ comes from Lemma~\ref{Lem:solvproper}.
\end{proof}

Finally, we classify the constructors of $G_{D,D'}/H$, namely, we have:

\begin{prop}\label{Prop:structurethm}
 For a connected subgroup $L\subset G_{D,D'}$, the following conditions are equivalent.
 \begin{eqenum}
  \item The subgroup $L$ is a constructor of $G_{D,D'}/H$.
  \item There exist $C\in M(n,\R)$ and $w\in\fh$ satisfying $CD'C=D$ and $L=L_{C,w}$.
 \end{eqenum}
\end{prop}

\begin{proof}
First, we show the implication (a)$\imply$(b). We denote by $\fl$ the Lie algebra of $L$. By Proposition~\ref{Prop:solvablecri}, we have $\fg_{D,D'}=\fl\oplus \fh$, and so the condition (b) follows from Proposition~$\ref{Prop:decomposition}$.\par
 Next, we show the implication (b)$\imply$(a). Take $C\in M(n,\R)$ and $w\in\fh$ satisfying $CD'C=D$. We are enough to show the properness and cocompactness of the $L_{C,w}$-action. By Lemma~\ref{Lem:LCCI}, the $L_{C,w}$-action on $G_{D,D'}/H$ is proper. By Note~\ref{Note:lcwdecomposition}, we have $\fg_{D,D'}=\fl_{C,w}\oplus\fh$, which implies the cocompactness by Proposition~\ref{Prop:solvablecri}.
\end{proof}

\subsection{Uniform lattices of $L_C$ and $L_{C,w}$}

In this subsection, we discuss necessary conditions for the existence of a uniform lattice in $L_C$ and $L_{C,w}$, namely, we show the following two propositions.

\begin{prop}\label{Prop:lclattice}
 Assume there exists $\ell\in G_{D,D'}-H_n$ such that the subgroup $L_C$ has an $\A_{\ell}$-invariant uniform lattice. Then the subspace $\fl'_C$ is $W_{t_0}$-invariant and we have $\det(W_{t_0}|_{\fl'_C})=\pm1$, where $t_0:=\pr_1(\ell)\in\R^\times$. 
\end{prop}

\begin{prop}\label{Prop:tracecondition}
 Suppose $C\in M(n,\R)$ and $w\in \fh$ satisfy the condition in Proposition~\ref{Prop:subalgebracri}. If the subgroup $L_{C,w}$ has a uniform lattice, then the condition $\operatorname{tr}{D'C}=0$ holds.
\end{prop}

To prove these propositions, we define a solvable Lie group.

\begin{defi}\label{def:SM}
 For $M\in M(m,\R)$, we consider the $\R$-action on $\R^m$, $\phi:\R\to GL(m,\R),\ t\mapsto \exp tM$, and denote by $S_M$ the semidirect product $\R\ltimes_{\phi}\R^m$. We denote by $\pr_\R:S_M\to \R$ the first projection. We regard $\R$ as a subgroup of $S_M$ by the injection $\R\to S_M,\ t\mapsto (t,0)$.
\end{defi}

We see some basic properties of $S_M$ (Lemma~\ref{Lem:centerlizer}, Note~\ref{Note:LCwform} and Proposition~\ref{Prop:SMlatticecondition}). 

\begin{lem}\label{Lem:centerlizer}
 Let $M\in M(m,\R)$ and $t_0\in\R$. If $\exp t_0 M$ does not have an eigenvalue $1$, then we have $\mathcal{Z}_{S_M}(t_0)=\R$.
\end{lem}

\begin{proof}
 For an element $(t_1,v_1)\in S_M$, we have:
 \[
  (t_1,v_1)\in\mathcal{Z}_{S_M}(t_0) \siff \A_{(t_0,0)}(t_1,v_1)=(t_1,v_1)\siff (\exp t_0M) v_1 =v_1\siff v_1=0.
 \]
\end{proof}

\begin{note}\label{Note:LCwform}
 Some Lie groups in this paper are isomorphic to $S_M$. Let $C\in M(n,\R)$ and $w\in\fh$ satisfy the condition of Proposition~\ref{Prop:subalgebracri}.
 \begin{itemize}
  \item  If $C$ is symmetric, $L_{C,w}\simeq S_M$, where $M:= \begin{pmatrix}
  D'C&0\\ -w^T&0
 \end{pmatrix}$,
  \item $L_{C,w}/\mathcal{Z}_{H_n}\simeq S_{D'C}$,
  \item $G_{D,D'}/\mathcal{Z}_{H_n}\simeq S_W$.
 \end{itemize}
\end{note}

\begin{prop}\label{Prop:SMlatticecondition}
For $M\in M(m,\R)$, we have $\tr M=0$ if the group $S_M$ has a uniform lattice.
\end{prop}

To prove this proposition, we prepare some facts and a proposition.

\begin{fact}[{\cite[Theorem~3.3]{R}}]\label{Fact:nilrad}
 Let $G$ be a connected solvable Lie group and $N$ its maximum connected (closed) normal nilpotent subgroup. Let $H$ be a cocompact closed subgroup of $G$. Assume that $H \cap N$ contains no non-trivial connected (closed) Lie subgroups which are normal in $G$. Then $N/(H\cap N)$ is compact.
\end{fact}

\begin{fact}[\cite{repofnilp}]\label{Fact:nilpcommutativity}
 Let $G$ be a 1-connected nilpotent Lie group and $\Gamma\subset G$ a uniform lattice. Then $G$ is commutative if and only if so is $\Gamma$.
\end{fact}

\begin{prop}\label{Prop:invlattice}
 For $M\in M(m,\R)$, we consider the following conditions.
 \begin{enumerate}
  \item There exists a cocompact discrete subgroup $\Gamma\subset \R^m$ satisfying $M\Gamma\subset \Gamma$.
  \item There exists a cocompact discrete subgroup $\Gamma\subset \R^m$ satisfying $M\Gamma= \Gamma$.
  \item All the coefficients of the characteristic polynomial of $M$ are integers.
  \item All the coefficients of the characteristic polynomial  of $M$ are integers and $\det M=\pm1$.
 \end{enumerate}
Then the implications (1)$\imply$(3) and (2)$\imply$(4) hold. Moreover, the equivalences (1)$\siff$(3) and (2)$\siff$(4) also hold if the eigenvalues of $M$ are distinct.
\end{prop}

\begin{proof}
 First, we prove the implication (1)$\imply$(3). Since the uniform lattice of $\R^n$ is isomorphic to $\Z^n$, $M$ is similar to an element of $M(n,\Z)$. Therefore, all the coefficients of characteristic polynomial are integers. \par
 Next, we show the implication (2)$\imply$(4). Since $M$ is invertible and both $\det M$ and $\det M\inv$ are integers, then we have $\det M=\pm1$.\par
 Finally, we prove the inverse implications (3)$\imply$(1) and (4)$\imply$(2). Assume the eigenvalues of $M$ are distinct, and all the coefficients of the characteristic polynomial of $M$ are integers. Since the eigenvalues of $M$ are distinct, there exists $v\in\R^n$ such that $(M^i v)_{i=0,1,\cdots,n-1}$ is a basis of $\R^n$. Then $\Gamma\coloneqq \xspanset{\Z}{M^i v}{i=0,1,\cdots,n-1}$ satisfies the condition (1). Actually, by Cayley--Hamilton's theorem, $M^n v$ is written as an linear combination of $(M^iv)_{i=0,1,\cdots,n-1}$ with integer coefficients, so we have $M^n v\in \Gamma$. Especially, in the case $\det M=\pm1$, since $M\inv v$ is also written as a linear combination of $(M^i v)_{i=0,1,\cdots,n-1}$, we have $M\Gamma = \Gamma$.
\end{proof}

\begin{ex}\label{ex:lattice}
Put $M:=\diag{1,-1,0}$ and let us see that the group $S_M$ admits a uniform lattice. Set $t_0:=\log(2+\sqrt{3})$. Since the characteristic polynomial of $\exp t_0 M$ is $t^3-5t^2+5t-1$ and its roots are distinct, there exists a uniform lattice $\Gamma_0\subset \R^3$ which is $\exp t_0M$-invariant by Proposition~\ref{Prop:invlattice}. The subgroup $\langle t_0\rangle\Gamma_0\subset S_M$ is a uniform lattice.
\end{ex}

\begin{proof}[{\bf Proof of Proposition~\ref{Prop:SMlatticecondition}}]
 If $M$ is nilpotent, we already have $\tr M=0$. On the other hand, if $M$ is not nilpotent, $\R^n$ is the maximum connected normal nilpotent subgroup of $S_M$. In this case, let $\Gamma$ be a uniform lattice of $S_M$. Then $\Gamma\cap \R^n$ is a uniform lattice of $\R^n$ by Fact~\ref{Fact:nilrad}. By Proposition~\ref{Prop:invlattice}, we have $\det(\exp tM)=\pm1$ for some $t\in\pr_\R(\Gamma)-\{0\}$. Then we obtain $\tr M=0$.
\end{proof}

Finally, we prove Propositions~\ref{Prop:lclattice} and \ref{Prop:tracecondition}.

\begin{proof}[{\rm \bf  Proof of Proposition~\ref{Prop:lclattice}}]
Let $\Gamma$ be an $\A_\ell$-invariant uniform lattice in $L_C$. Since $L_C$ is $\A_\ell$-invariant, the subspace $\fl'_{C}$ is $W_{t_0}$-invariant by Note~\ref{Note:easycalc}. Now we are enough to show $\det(W_{t_0}|_{\fl'_C})=\pm1$. We denote by $\log:H_n\to \fh_n$ the inverse of the exponential map $\exp:\fh_n\to H_n$ (diffeomorphism).\\
 (a) The case where $L_C$ is commutative.\par
 Since $\log(\Gamma)$ is an $\Ad_\ell$-invariant uniform lattice in $\fl_C$, we obtain $\det (\Ad_\ell|_{\fl_C})=\pm1$  by Proposition~\ref{Prop:invlattice}, and so  $\det (W_{t_0}|_{\fl'_C})=\det (\Ad_\ell|_{\fl_C})=\pm1$.\\
 (b) The case where $L_C$ is not commutative.\par
 By Fact~\ref{Fact:nilpcommutativity}, the uniform lattice $\Gamma$ is also non-commutative. Since the quotient $\mathcal{Z}_{H_n}/(\Gamma\cap \mathcal{Z}_{H_n})$ is compact, we apply Property \ref{Property:centerquotient} to the natural surjection $\pi:L_C\to L_C/\mathcal{Z}_{H_n}$, then $\widetilde{\Gamma}\coloneqq\pi(\Gamma)$ is a discrete subgroup of $L_C/\mathcal{Z}_{H_n}\simeq \R^{2n}$. Hence, the subset $\log(\widetilde{\Gamma})$ is a $W_{t_0}$-invariant uniform lattice of $\fl_C/\fz$. Then we have $\det (W_{t_0}|_{\fl'_C})=\det(W_{t_0}|_{\fl/\fz})=\pm 1$ by Proposition~\ref{Prop:invlattice}. 
\end{proof}

\begin{proof}[{\rm \bf  Proof of Proposition~\ref{Prop:tracecondition}}]
Let $\Gamma$ be a uniform lattice in $L_{C,w}$.\\
 {(a) The case where $L_C$ is commutative.}\par
Since the matrix $C$ is symmetric by Note~\ref{Note:structurelcw}, the condition $\tr D'C=0$ follows from Note~\ref{Note:LCwform} and Proposition~\ref{Prop:SMlatticecondition}.\\
 {(b) The case where $L_C$ is not commutative.}\par
 Since $L_C$ is the maximum connected normal nilpotent Lie subgroup of $L_{C,w}$, the subgroup $\Gamma\cap L_C$ is a uniform lattice of $L_C$ by Fact~\ref{Fact:nilrad}. By Fact~\ref{Fact:nilpcommutativity}, the lattice $\Gamma\cap L_C$ is not commutative, either. Hence $\mathcal{Z}_{H_n}/(\Gamma\cap \mathcal{Z}_{H_n})$ is compact. By applying Property \ref{Property:centerquotient} to the natural surjection $\pi:L_{C,w}\to L_{C,w}/\mathcal{Z}_{H_{n}}$, we find $\widetilde{\Gamma}:=\pi(\Gamma)$ is a uniform lattice of the group $\widetilde{L_{C,w}}:=\pi(L_{C,w})$. Since we have $\widetilde{L_{C,w}}\simeq S_{D'C}$ by Note~\ref{Note:LCwform}, the condition $\operatorname{tr}D'C=0$ follows from Proposition~\ref{Prop:SMlatticecondition}.
\end{proof}

\subsection{Criterions of the existence of compact Clifford--Klein forms}\label{Subsection:Criterions}

In this subsection, we give the following criterions for the existence of compact Clifford--Klein forms of $G_{D,D'}/H$. If $D$ and $D'$ are diagonal, this proposition also follows from \cite[Proposition~4.8]{OKCK}.

\begin{prop}\label{Prop:mainprop}
The following conditions are equivalent.
 \begin{eqenum}
  \item The symmetric space $G_{D,D'}/H$ admits compact Clifford--Klein forms.
  \item There exists $C\in M(n,\R)$ satisfying the following conditions.
  \begin{enumerate}
   \renewcommand{\labelenumii}{(\roman{enumii})}
   \item The matrix $A_t +B_t C$ is invertible for any $t\in\R$.
   \item The subgroup $L_C$ has an $\A_{\ell}$-invariant uniform lattice for some $\ell\in G_{D,D'}-H_n$.
  \end{enumerate}
 \end{eqenum}
 Moreover, the following condition is a necessary condition of the above condition~(ii).
 \begin{enumerate}
  \renewcommand{\labelenumi}{(\roman{enumi}')}
  \setcounter{enumi}{1}
  \item The subspace $\fl'_C$ is $W_{t_0}$-invariant and $\det (W_{t_0}|_{\fl'_C})=\pm1$ for some $t_0\in\R^\times$.
 \end{enumerate}
\end{prop}

We prove this in Subsection~\ref{Subsection:Lcondition}.
If $G_{D,D'}$ is completely solvable, we have an easier criterion.

\begin{prop}\label{Prop:csmainprop}
Assume $G_{D,D'}$ is completely solvable. The following conditions are equivalent.
 \begin{eqenum}
  \item The symmetric space $G_{D,D'}/H$ admits compact Clifford--Klein forms.
  \item There exist $C\in M (n,\R)$ and $w\in\fh$ satisfying the following conditions.
  \begin{enumerate}
   \renewcommand{\labelenumii}{(\roman{enumii})}
   \item $CD'C=D$.
   \item The subgroup $L_{C,w}$ admits a uniform lattice.
  \end{enumerate}
 \end{eqenum}
 Moreover, the following condition is a necessary condition of the above condition~(ii). 
 \begin{enumerate}
  \renewcommand{\labelenumi}{(\roman{enumi}')}
  \setcounter{enumi}{1}
  \item $\tr D'C=0$.
 \end{enumerate}
\end{prop}

\begin{proof}
The implication (b)(ii)$\imply$(ii') follows from Proposition~\ref{Prop:tracecondition}, so we are enough to show the equivalence (a)$\siff$(b), which equivalent to check that the following conditions are equivalent:
 \begin{enum}{\Alph}
  \item there exists a discrete subgroup $\Gamma \subset G_{D,D'}$ which acts on $G_{D,D'}/H$ properly discontinuously, cocompactly and freely,
  \item there exists a constructor $L$ of $G_{D,D'}/H$ and $L$ has a uniform lattice,
  \item there exists $C\in M(n,\R)$ and $w\in \fh$ such that $CD'C=D$ and $L_{C,w}$ has a uniform lattice.
 \end{enum}
 The implication (B)$\Rightarrow$(A) comes from Note~\ref{note:properisfree} (2). Since $G_{D,D'}$ is completely solvable, the implication (A)$\Rightarrow$(B) follows from Fact~\ref{Fact:contianalogy}. The equivalence (B)$\siff$(C) follows from Proposition~\ref{Prop:structurethm}.
\end{proof}

\subsection{Intermediate syndetic hulls}

In this subsection, we consider the existence problem of compact Clifford--Klein forms for solvable homogeneous spaces. A difficulty arises when $G_{D,D'}$ is not completely solvable. In fact, in this case, a discrete subgroup $\Gamma\subset G_{D,D'}$ may fail to have its syndetic hull (Remark~\ref{Rem:kobconj}). To overcome this difficulty, we introduce intermediate syndetic hulls which play a similar role to syndetic hulls.

\begin{defi}[intermediate syndetic hull]\label{Def:LSH}
For a closed subgroup $\Gamma\subset G_{D,D'}$, a closed subgroup $L'\subset G_{D,D'}$ is called an {\it intermediate syndetic hull} of $\Gamma$ if $L'$ satisfies the following conditions with $L_0:=L'\cap H_n$:
 \begin{enumerate}
  \renewcommand{\labelenumi}{(\roman{enumi})}
  \item $L_0$ is connected,
  \item there exists $\ell\in \Gamma-H_n$ satisfying $L'=\gen{\ell} L_0$, 
  \item $\Gamma$ is a cocompact subgroup of $L'$.
 \end{enumerate} 
\end{defi}

\begin{note}\label{Note:equivcond}
 The condition (iii) above is equivalent to:
\begin{enumerate}
 \setcounter{enumi}{2}
 \renewcommand{\labelenumi}{(\roman{enumi}')}
 \item $\Gamma$ is a subgroup of $L'$ and $\Gamma\cap H_n$ is cocompact in $L_0$.
\end{enumerate}
\end{note}


In this subsection, we show the next:

\begin{prop}\label{Prop:existenceofL'}
 Let $\Gamma\subset G_{D,D'}$ be a discrete subgroup acting on $G_{D,D'}/H$ cocompactly. Then $\Gamma$ has an intermediate syndetic hull.
\end{prop}

To prove this proposition, we use Lemmas~\ref{Lem:discreteandL'} and \ref{Lem:discrete}.

\begin{lem}[Criterion of the existence of an intermediate syndetic hull]\label{Lem:discreteandL'}
Let $\Gamma\subset G_{D,D'}$ be a discrete subgroup and put $\Gamma_0:=\Gamma\cap H_n$. Then the following conditions are equivalent.
 \begin{eqenum}
  \item The subgroup $\pr_1(\Gamma)\subset \R$ is non-trivial and discrete.
  \item There exists $\gamma\in \Gamma-H_n$ satisfying $\Gamma=\langle \gamma\rangle \Gamma_0$.
  \item $\Gamma$ has an intermediate syndetic hull.
 \end{eqenum}
\end{lem}

\begin{proof}
 The implication (c)$\imply$(a) is easy. We show the implication (a)$\imply$(b).
 Let $t_0\in \pr_1(\Gamma)$ be a generator of $\pr_1(\Gamma)$. We take $\gamma\in\Gamma$ satisfying $\pr_1(\gamma)=t_0$. Then the following exact sequence splits by the group homomorphism $s:\pr_1(\Gamma)\to \Gamma,\quad t_0\mapsto \gamma$. 
\[
 \xymatrix{\{e\}\ar[r]&\Gamma_0\ar[r]&\Gamma\ar[r]^{\pr_1}&\pr_1(\Gamma)\ar[r] \ar @/^4mm/[l]_{s}&\{e\}}.
\]
Then we have $\Gamma=s(\pr_1(\Gamma))\Gamma_0=\langle\gamma\rangle\Gamma_0$ and so the condition (b).\par
Then we see the implication (b)$\imply$(c). By Fact~\ref{Fact:solvablesyndetichull}, the discrete subgroup $\Gamma_0$ has the syndetic hull $L_0\subset H_n$. Then the closed subgroup $L'=\gen{\gamma}L_0$ clearly satisfies the conditions (i) and (ii) in Definition~\ref{Def:LSH} and (iii') in Note~\ref{Note:equivcond}.
\end{proof}

In the rest of this subsection, we use the following:

\begin{notationandsetting}
 Let $N$ be a 1-connected 2-step nilpotent Lie group, $\fn$ its Lie algebra.
 \begin{itemize}
  \item $\fn_{\C}:=\fn\otimes \C$
  \item $\fn(A,\lambda)\subset \fn_\C$ is the generalized eigenspace of $A\in \operatorname{End}(\fn)$ with respect to an eigenvalue $\lambda\in \C$.
  \item $\fn(A,\lambda)^\perp:=\bigoplus_{k\neq \lambda}\fn(A,k)$.
 \end{itemize}   
\end{notationandsetting}

\begin{lem}\label{Lem:discrete}
  Let $M\in \Der \fn$. we put $\phi:\R\to \Aut(\fn),\ t\mapsto \exp tM$. By identifying $\Aut(N)\simeq \Aut(\fn)$, put $G:=\R\ltimes_\phi N$ and $\pr_\R:G\to \R$ the first projection. Let $\Gamma\subset G$ be a discrete subgroup. Then $\pr_\R(\Gamma)$ is discrete if the following conditions are satisfied.
 \begin{enumerate}
  \item $\fn(M,0)=[\fn_\C,\fn_\C]$.
  \item $\Gamma_0:=\Gamma\cap N$ is commutative.
  \item $\fl_0\not \subset [\fn,\fn]$, where $\fl_0$ is the Lie algebra of the syndetic hull of $\Gamma_0$. 
 \end{enumerate}
\end{lem}

To prove this lemma, we use the following lemma.

\begin{lem}\label{Lem:discretelem}
 In the setting in Lemma~\ref{Lem:discrete} with the assumptions (1), (2) and (3), put a finite set $F$ as follows, where $\lambda_k\in\C$ are eigenvalues of $M$. 
 \[
  F:= \set{f_I(t):=\ds\sum_{k\in I} e^{t\lambda_{k}}}{I\subset\{1,2,\cdots,\dim \fn\},\ f_I \text{ is not constant}}\subset C^\infty(\R).
 \]
 Fix $\gamma=(t_0,\exp v_0)\in\Gamma$ and assume $\fn(\phi_{t_0},1)=[\fn_\C,\fn_\C]$. We think $\Ad_\gamma\in GL(\fn_\C)$. Then there exists $f\in F$ satisfying $\tr \Ad_\gamma|_{\fl_0}=f(t_0)$.
\end{lem}

\begin{proof} 
Since $\fn$ is 2-step nilpotent, we have:
 \[
  \Ad_\gamma X=\phi_{t_0}X+[v_0,\phi_{t_0}X]\quad(\forall X\in\fn_\C).
 \]
Put $V:=\fn(\phi_{t_0},1)^\perp$, then we have a decomposition $\fn_\C=V\oplus [\fn_\C,\fn_\C]$ and the following matrix representation:
\[
\Ad_\gamma|_{\fn_\C}=
 \begin{pmatrix}
  \phi_{t_0}|_V&0\\
  *&{\phi_{t_0}}|_{[\fn_\C,\fn_\C]}
 \end{pmatrix}.
\]
 Therefore, the eigenvalues of $\Ad_\gamma|_{\fn_\C}$ coincides with them of $\phi_{t_0}$. Since $\fl_0$ is $\Ad_\gamma$-invariant, there exists $I\subset \{1,2,\cdots,\dim \fn\}$ such that $\tr \Ad_\gamma|_{\fl_0}=f_I(t_0)$. Then we only have to show that $f_I$ is not constant.
It is enough to show that there exists $k\in I$ satisfying $\lambda_k\neq 0$. Assume $\lambda_k=0$ for any $k\in I$ then the eigenvalues of $\Ad_\gamma|_{\fl_0}$ are all 1. Then by the matrix representation $\Ad_\gamma|_{\fn_\C}$, we have $\fl_0\otimes \C\subset [\fn_\C,\fn_\C]$, which contradicts the condition $\fl_0\not\subset [\fn,\fn]$. Therefore, we obtain $f_I\in F$.   
\end{proof}

\begin{proof}[{\rm \bf Proof of Lemma~\ref{Lem:discrete}.}]
It is enough to show that $\pr_\R(\Gamma)$ is included in a countable and closed subset of $\R$. 
We put a finite set $F$ as in Lemma~\ref{Lem:discretelem}, and put subsets $A, B\subset \R$ as follows:
 \begin{align*}
  A&:=\set{t\in\R}{\fn(\phi_t,1)=[\fn_\C,\fn_\C]},\\
  B&:=\bigcup_{f\in F} f\inv(\Z).
 \end{align*}
 Then $A^c:=\R- A$ and $B$ are countable and closed.
Take any $\gamma= (t,v)\in\Gamma$. We are enough to show that $t\in A^c\cup B$. Assume $t\in A$, then there exists $f\in F$ such that $\tr \Ad_{\gamma}|_{\fl_0}= f(t)$ by Lemma~\ref{Lem:discretelem}. Since the subalgebra $\fl_0$ is abelian and has an $\Ad_\gamma$-invariant uniform lattice, $\tr \Ad_\gamma|_{\fl_0}$ must be integer by Proposition~\ref{Prop:invlattice}. Therefore we have $t\in f\inv (\Z)\subset  B$. 
\end{proof}

Finally, we prove Proposition~\ref{Prop:existenceofL'}.

\begin{proof}[{\rm \bf Proof of Proposition~\ref{Prop:existenceofL'}}]
Take a discrete subgroup $\Gamma\subset G_{D,D'}$ acting on $G_{D,D'}/H$ cocompactly. By Lemma~\ref{Lem:discreteandL'}, we are enough to show that $\pr_1(\Gamma)$ is discrete. 
We consider the natural surjection $\pi:G_{D,D'}\to G_{D,D'}/\mathcal{Z}_{H_n}\simeq S_W$ (see Note~\ref{Note:LCwform}) and put $\widetilde{G_{D,D'}},\ \widetilde{H}$ and $\widetilde{\Gamma}$ the image by $\pi$ of $G_{D,D'},H$ and $\Gamma$, respectively. By Property~\ref{Property:centerquotient}, $\widetilde{\Gamma}$ acts on $\widetilde{G_{D,D'}}/\widetilde{H}$ cocompactly. We consider $\wt{\Gamma}$ as a subgroup of $S_W$. We put $\Gamma_0:=\Gamma\cap H_n$. 
\begin{enumerate}
 \renewcommand{\labelenumi}{(\Alph{enumi})}
 \item The case $\Gamma_0\subset \mathcal{Z}_{H_n}$.\par We are enough to show $\pr_1(\Gamma)\subset A$ for $A:=\set{t\in\R}{W_t \text{ has an eigenvalue }1}$. Here, note that $A\subset \R$ is closed and countable. Assume $\pr_1(\Gamma)\not\subset A$. Then there exists $(t_0,v_0)\in \widetilde{\Gamma}$, where $t_0\in \R-A$ and $v_0\in\R^{2n}$. By $[\Gamma,\Gamma]\subset \Gamma_0\subset \mathcal{Z}_{H_n}$, $\widetilde{\Gamma}$ is abelian. Especially, we have $\widetilde{\Gamma}\subset \mathcal{Z}_{S_W}((t_0,v_0))=\A_{x}\mathcal{Z}_{S_W}(t_0)$, where $x:=(\id-W_{t_0})\inv v_0$. By Lemma~\ref{Lem:centerlizer}, we have:
       \[
	\widetilde{\Gamma}\subset \A_{x}\mathcal{Z}_{S_W}(t_0)\subset \A_{x} \R. 
       \]
       On the other hand, since the $\R$-action on $\widetilde{G_{D,D'}}/\widetilde{H}$ is not cocompact, neither is the $\widetilde{\Gamma}$-action, which contradicts the cocompactness of $\widetilde{\Gamma}$-action (Property~\ref{Property:centerquotient}). Then we have $\pr_1(\Gamma)\subset A$.
 \item The case $\Gamma_0\not\subset \mathcal{Z}_{H_n}$.
       \begin{enumerate}
	\item $\Gamma_0$ is commutative.\\ We denote by $\fl_0$ the Lie algebra of the syndetic hull of $\Gamma_0$. By using Lemma~\ref{Lem:discrete} for $(N,M,\Gamma)=(H_n,\diag{W,0},\Gamma)$, we get $\pr_1(\Gamma)$ is discrete. Actually, it is easy to check the conditions in Lemma~\ref{Lem:discrete} as follows. 
	      \begin{enumerate}
	       \renewcommand{\labelenumiii}{(\arabic{enumiii})}
	       \item $\fn(M,0)=\fz\otimes \C(=[\fh_n,\fh_n]\otimes \C)$.
	       \item $\Gamma_0$ is commutative by the assumption (a).
	       \item By the assumption $\Gamma_0\not \subset\mathcal{Z}_{H_n}$, we have $\fl_0\not\subset [\fh_n,\fh_n]=\fz$.
	       \end{enumerate}	      
	\item $\Gamma_0$ is non-commutative.\\ We have $[\Gamma_0,\Gamma_0]$ is non-trivial and so $\mathcal{Z}_{H_n}/(\mathcal{Z}_{H_n}\cap \Gamma_0)$ is compact. By Property~\ref{Property:centerquotient}, $\widetilde{\Gamma}$ is a discrete subgroup of $\widetilde{G_{D,D'}}$ and acts on $\widetilde{G_{D,D'}}/\widetilde{H}$ cocompactly. By using Lemma~\ref{Lem:discrete} for $(N,M,\Gamma)=(\R^{2n},W,\widetilde{\Gamma})$, we have the discreteness of $\pr_1(\Gamma)=\pr_\R(\wt{\Gamma})$. Actually, it is easy to check the conditions in Lemma~\ref{Lem:discrete} as follows.
	      \begin{enumerate}
	       \renewcommand{\labelenumiii}{(\arabic{enumiii})}
	       \item $\fn(M,0)=\{0\}(=[\R^{2n},\R^{2n}]\otimes\C)$.
	       \item Since $N$ is commutative, so is $\wt{\Gamma_0}:=\wt{\Gamma}\cap N$.
	       \item By the assumption $\Gamma_0\not \subset\mathcal{Z}_{H_n}$, we have $\wt{\fl_0}\not\subset [\R^{2n},\R^{2n}]=\{0\}$, where $\wt{\fl_0}$ is the Lie algebra of the syndetic hull of $\wt{\Gamma_0}$.
	       \end{enumerate}
       \end{enumerate}
\end{enumerate}
\end{proof}

\subsection{(L) condition}\label{Subsection:Lcondition}

In this subsection, our goal is to prove Proposition~$\ref{Prop:mainprop}$. To do this, we introduce (L) condition.

\begin{defi}[(L) condition]\label{Def:lcondition}
 We say a closed subgroup $L'\subset G_{D,D'}$ satisfies (L) {\it condition} if the following conditions are satisfied.
 \begin{enumerate}
  \item $L'$ is unimodular,
  \item $L_0 = L' \cap H_{n}$ is connected,
  \item there exists $\ell\in L'-L_0$ satisfying $L' = \langle \ell\rangle L_0 $.
 \end{enumerate}
\end{defi}

Clearly, intermediate syndetic hulls satisfy (L) condition.
We see a fundamental property of (L) condition, namely, we have:
\begin{prop}[Criterion of properness and cocompactness]\label{Prop:criofcompactproper}
Suppose a closed subgroup $L'\subset G_{D,D'}$ satisfies (L) condition and put $L_0:=L'\cap H_n$. Then the following conditions are equivalent:
\begin{eqenum}
 \item the $L'$-action on $G_{D,D'}/H$ is proper and cocompact,
 \item the pair $(L_0,H)$ satisfies the property (CI) in $G_{D,D'}$ and the $L_0$-action on $H_n/H$ is cocompact,
 \item there exists a matrix $C\in M(n,\R)$ such that $L_0=L_C$ and $A_t + B_t C$ is invertible for any $t\in \R$ (see Notation \ref{notation:second}).
\end{eqenum}
\end{prop}

\begin{proof}
 Take $\ell\in L'$ satisfying $L'=\gen{\ell}L_0$. First we show the equivalence (a)$\siff$(b). Since the $L'$-action on $G_{D,D'}/H_n$ is cocompact, it follows from Propositions~\ref{Prop:propciequiv} and \ref{Prop:criofccp} by putting $(G,N,L_0,L_1,H)=(G_{D,D'},H_n,L_0,\gen{\ell},H)$. Here we need to check that the tuple satisfies the condition of Setting~\ref{set:propersetting}. The condition $L_1\subset \mathcal{N}_{G}(L_0)$ is clear. Take $T\in\fg_{D,D'}$ satisfying $\exp T=\ell$, we have $\exp \R T\pitchfork H$ in $G_{D,D'}$ by Lemma~\ref{Lem:solvproper}, and so $L_1\pitchfork N$.\par
 Next we show the equivalence (b)$\siff$(c). By Proposition~\ref{Prop:propciequiv} for $G=N=H_n$, the condition that $(L_0,H)$ satisfies the property (CI) in $G_{D,D'}$ implies $L_0\pitchfork H$ in $H_n$. Then the equivalence (b)$\siff$(c) follows from Propositions~\ref{Prop:solvablecri}, \ref{Prop:decomposition} and \ref{Prop:properequiv}.
\end{proof}

Finally, we prove Proposition~\ref{Prop:mainprop}. 
\begin{proof}[\rm \bf Proof of Proposition~\ref{Prop:mainprop}]
 First, we show the implication (a)$\imply$(b). Take a discrete subgroup $\Gamma\subset G_{D,D'}$ such that $\Gamma\bs G_{D,D'}/H$ is a compact Clifford--Klein form. By Proposition~\ref{Prop:existenceofL'}, we take an intermediate syndetic hull $L'\subset G_{D,D'}$ of $\Gamma$. Put $L_0:=L'\cap H_n$ and take $\ell\in\Gamma-H$ satisfying $L'=\gen{\ell}L_0$. By Fact~\ref{Fact:contianalogy}, the $L'$-action on $G_{D,D'}/H$ is proper and cocompact. Take $C\in M(n,\R)$ such that $L_0=L_C$ and $A_t+B_tC$ is invertible for any $t\in\R$ by Proposition~\ref{Prop:criofcompactproper}. Hence the condition (b)(i) holds.
 Moreover, by Proposition~\ref{Prop:existenceofL'}(iii), $L_C$ admits an $\A_\ell$-invariant uniform lattice $\Gamma\cap H_n$ and so the condition (b)(ii) holds. \par
 Next, we check the implication (b)$\imply$(a). Take $C\in M(n,\R)$ and $\ell\in G_{D,D'}-H_n$ satisfying the condition (b) and let $\Gamma_0\subset L_C$ be an $\A_\ell$-invariant uniform lattice. Put $L'\coloneqq\langle \ell\rangle L_C$ and $\Gamma\coloneqq \langle \ell\rangle \Gamma_0$. Note that $L'$ satisfies the (L) condition and $\Gamma$ is a uniform lattice of $L'$. The $L'$-action is proper and cocompact by Proposition~\ref{Prop:criofcompactproper}. 
 Therefore, $\Gamma\bs G_{D,D'}/H$ is a compact Clifford--Klein form by Note~\ref{note:properisfree}.\par
 Finally, the implication (b)(ii)$\imply$(ii') comes from Proposition~\ref{Prop:lclattice}.
\end{proof}

\section{Proof of the main theorem}\label{Section:Proof}

In this section, we give a proof of the main theorem (Theorem~\ref{Thm:mainthm}).



We check the existence of compact Clifford--Klein forms for the spaces with signature $(2,2)$ which correspond to each case in Fact~\ref{Fact:classification}.

\subsection{On the spaces which correspond to the triples $(\fg_{\rm nil},\sigma, g_\pm)$}

We consider the spaces which correspond to Case~(I) in Fact~\ref{Fact:classification}.

\subsubsection{transvection group}

Note that $G_{\rm nil}$ is the transvection group of the symmetric spaces $(G_{\rm nil}/H,\sigma, g_\pm)$.

\begin{prop}
 The symmetric space $G_{\rm nil}/H$ does not admit compact Clifford--Klein forms.
\end{prop}

\begin{proof}
 Assume $G_{\rm nil}/H$ admits a compact Clifford--Klein form $\Gamma\bs G_{\rm nil}/H$. Since the Lie group $G_{\rm nil}$ is 1-connected and nilpotent, there exists a constructor $L$ including $\Gamma$ cocompactly by Fact~\ref{Fact:solvablesyndetichull}. By Proposition~\ref{Prop:solvablecri}, we get $\fg_{\rm nil}=\fl\oplus\fh$, where $\fl$ and $\fh$ are the Lie algebras of $L$ and $H$, respectively. On the other hand, there is no such a subalgebra $\fl$ by Lemma~\ref{Lem:nosubalgebra}.
\end{proof}

In this proof, we use the following lemmas.

\begin{lem}\label{Lem:nosubalgebra}
 For $\fg_{\text{nil}}$ and $\fh$ in Definition~\ref{Def:gnil}, no subalgebra $\fl\subset \fg_{\rm nil}$ satisfies $\fg_{\rm{nil}}=\fl\oplus\fh$.
\end{lem}

\begin{proof}
 Take $\fq$ as in Definition~\ref{Def:gnil}. Suppose a subalgebra $\fl\subset \fg_{\rm nil}$ satisfies $\fg_{\rm{nil}}=\fl\oplus\fh$. We denote by $\fz_0$ the center of $\fg_{\rm nil}$. Since the element $C_1,C_2\in\fz_0\cap \fq$ satisfy the assumption of Lemma~\ref{Lem:zinl} below, we have $\fz_0\subset \fl$. Take $k_1,k_2\in\R$ satisfying $A_1+k_1B, A_2+k_2B\in\fl$.
Then $[A_1+k_1B,A_2+k_2B]-B=k_1C_1+k_2C_2\in\fz\subset \fl$ and so we have $B\in \fl$, which is a contradiction.
\end{proof}

\begin{lem}\label{Lem:zinl}
 Let $\fg$ be a Lie algebra, $\fz_0$ its center and $\fg=\fq\oplus\fh=\fl\oplus\fh$ its decompositions as a linear space. Assume $\fh$ and $\fl$ are subalgebras of $\fg$ and $\fh$ is abelian. For $Z_0\in\fz_0\cap \fq$ satisfying the following condition, we have $Z_0\in\fl$.
 \[
  \forall Y\in(\fh-\{0\}),\ \exists X\in\fq \st [X,Y]=Z_0.
 \]
\end{lem}

\begin{proof}
 Take the linear map $\phi:\fq\to\fh$ satisfying $\fl=\set{x+\phi(x)}{x\in\fq}$. Since $\phi(Z_0)=0$ implies $Z_0\in\fl$, we assume $\phi(Z_0)\neq 0$. Then there exists $X\in\fq$ satisfying $[X,\phi(Z_0)]=Z_0$, so we have $\fl\ni[X+\phi(X),Z_0+\phi(Z_0)]=[X,\phi(Z_0)]=Z_0$.
\end{proof}

\subsubsection{isometry group}

First, we calculate isometry group of the symmetric spaces.

\begin{prop}
 We have a Lie group isomorphism $\Gi\simeq (\Z/2\Z\ltimes O(1,1))\ltimes\Gt$. Especially, $\Gi$ is completely solvable.
\end{prop}

In the following, we calculate the group $\Hi$.

\begin{lem}\label{Lem:autolemnil}
 Using the basis $(A_1,A_2,B,C_1,C_2)$ of $\fg_{\rm nil}$, we have:
 \begin{align*}
 \Aut(\fg_{\rm nil},\sigma,g_\pm)&=\left\langle
\begin{pmatrix}
  I_{1,1}&&\\
  &-1&\\
  &&I_{1,1}\\
 \end{pmatrix},
  \begin{pmatrix}
  P&&\\
  &1&\\
  &&P\\
  \end{pmatrix},
\begin{pmatrix}
  I_2&&\\
  &1&\\
  aI_2&&I_2\\
 \end{pmatrix}
  :\begin{gathered}
  a\in\R,\\ P\in\gen{P_t,J,-I_2}
   \end{gathered}\right\rangle \\
  &\simeq (\Z/2\Z\ltimes O(1,1))\ltimes \R,
 \end{align*}
 where $P_t:=\exp (-tI_{1,1})=\diag{e^{-t_0},e^{t_0}}$ and $J:=\begin{pmatrix}
								&-1\\1
							       \end{pmatrix}$.
\end{lem}

\begin{proof}
  Let $\phi\in \Aut(\fgt,\sigma,g_\pm)$. We use $(A_1,A_2,B,C_1,C_2)$ as the basis of $\fgt$. Then the representation matrix of $\phi$ is:
 \[
\phi=
 \begin{pmatrix}
  P&&\\
  &k&\\
  Q&&R\\
 \end{pmatrix}\in GL(5,\R),
 \]
 where $k\in\R^\times$ and $P,Q,R\in GL(2,\R)$. Then we have:
 \begin{align}
  &\text{$g_\pm$ is $\phi$ invariant} \iff \phi^Tg_\pm\phi=g_\pm
  &&\iff  \begin{cases}P^TJR=J\\
	   Q^TJP+P^TJQ=0
	     \end{cases}\\
  &[\phi(A_1),\phi(A_2)]=\phi([A_1,A_2])&&\iff \det P=k,\label{3}\\
  &[\phi(A_i),\phi(B)]=\phi([A_i,B])\ (i=1,2)&&\iff kP=R, \label{4}
 \end{align}
By the first identity of (0.7) and (0,5), we have $kP^TJP=J$. Then we have $k^2(\det P)^2=1$ and by (0.6), we have $k^4=1$, and so $k=\pm1$. By the second identity of (0.5), $P^TJQ$ is skew symmetric. We put $a\in\R$ by $P^T JQ=aJ$, then we have $Q=akP$. Moreover, a direct calculation leads us $\phi\in\Aut(\fgt,\sigma,g)$ if $\phi$ satisfies the conditions of Lemma~\ref{Lem:autolemnil}.

\end{proof}

\begin{lem}\label{Lem:Htisomgnil}
 By the identification $\Hi\simeq \Aut(\fgt,\sigma,g)$, we have:
 \[
  \Ht=\set{  \begin{pmatrix}
	       I_2&&\\&1&\\a I_2&&I_2
	      \end{pmatrix}
}{a\in\R}\simeq \R.
 \]
\end{lem}

Then we prove:

\begin{prop}
 $\Gi/\Hi$ admits compact Clifford--Klein forms.
\end{prop}

\begin{proof}
 We are enough to show that $(\Gi)_0/(\Hi)_0$ admits compact Clifford--Klein forms. Note that $\fgi\simeq \R\ltimes \fgt$. Let $X$ be the generator of $\R$, we have:
 \[
  [X,A_1]=A_2,\ [X,A_2]=A_1,\ [X,B]=0,\ [X,C_1]=C_2,\ [X,C_2]=C_1.
 \]
 Then we put $T_1:=-A_1-A_2-X,\ T_2:=A_1-A_2+2B-2C_1-2C_2$ and:
 \[
  \fl:=\xspan{\R}{T_1,T_2,C_1,C_2},\ L:=\exp L.
 \]
 Note that $[T_1,T_2]=T_2$. Then $\fl$ is an ideal of $\fgi$ and we have $\fgi=\fl\oplus\fh$. We are enough to show that:
 \begin{enumerate}
  \item $L\pitchfork \Hi$ in $\Gi$,
  \item $L$ acts on $\Gi/\Hi$ cocompactly,
  \item $L$ admits a uniform lattice $\Gamma$.
 \end{enumerate}
 We check the condition~(1). Note that $\fgi$ is 3-step nilpotent. By Fact~\ref{Fact:nasrin}, it is equivalent to the condition that $(L,H)$ is (CI) in $\Gi$. By Note~\ref{Note:nilpci}, we are enough to check that $\Ad_g\fl\cap \fh=\{0\}$ for all $g\in \Gi$. Since $\fl$ is an ideal, we have $\Ad_g\fl\cap \fh=\fl\cap\fh=\{0\}$.
 The condition (2) follows from Proposition~\ref{Prop:solvablecri}.
  Finally we check the condition (3). We put:
 \[
  \fl_0:=\xspan{\Z}{T_1,T_2,C_1,C_2},\ \Gamma:=\exp \fl_0.
 \]
  Then $\Gamma$ is a lattice of $L$, which construct a compact Clifford--Klein form.
\end{proof}

\subsection{On the spaces which correspond to the triples $(\fg_{D,D'},\sigma,g)$}

We consider the spaces which correspond to Case~(II) in Fact~\ref{Fact:classification}. These spaces are written as $G_{D,D'}/H$ for some matrices $D,D'\in\sym^{(\text{reg})}(n,\R)$. 

\subsubsection{transvection group}

Note that the transvection group of $(G_{D,D'}/H,\sigma,g)$ is $G_{D,D'}$.
We check the existence of compact Clifford--Klein forms by using Propositions~$\ref{Prop:mainprop}$ and $\ref{Prop:csmainprop}$. To do this, we introduce subsets of $M(n,\R)$.

\begin{defi}
 We define the following sets:
 \begin{itemize}
  \item $\mathcal{P}_{D,D'}:=\set{C\in M(n,\R)}{\text{$C$ satisfies the condition (b) (i) of Proposition~} \ref{Prop:mainprop}}\\ \hspace{1cm} =\set{C\in M(n,\R)}{A_t+B_tC \text{ is invertible for any } t\in\R}$,
  \item $\mathcal{L}_{D,D'}:=\set{C\in M(n,\R)}{\text{$C$ satisfies the condition (b) (ii') of Proposition~} \ref{Prop:mainprop}}\\ \hspace{1cm} =\set{C\in M(n,\R)}{\exists t_0\in\R^\times\st \fl'_C \text{ is $W_{t_0}$-invariant and $\det(W_{t_0}|_{\fl'_C})=\pm1$.}}$,
  \item $\mathcal{P}^c_{D,D'}:=\set{C\in M(n,\R)}{\text{$C$ satisfies the condition (b) (i) of Proposition~} \ref{Prop:csmainprop}}\\ \hspace{1cm} =\set{C\in M(n,\R)}{CD'C=D}$,
  \item $\mathcal{L}^c_{D,D'}:=\set{C\in M(n,\R)}{\text{$C$ satisfies the condition (b) (ii') of Proposition~} \ref{Prop:csmainprop}}\\ \hspace{1cm} =\set{C\in M(n,\R)}{\tr D'C=0}.$
 \end{itemize}
\end{defi}

In this subsection, we put $n=2$ and denote the five-dimensional Heisenberg Lie algebra by $\fh_2=\xspan{\R}{X_1,X_2,Y_1,Y_2,Z}$ and the Heisenberg Lie group by $H_2$. We also use the notation $W,W_t\in M(2n,\R)$, $A_t,B_t\in M(n,\R)$ as in Notation~\ref{notation:second}.
 
\begin{rem}
 The condition $\mathcal{P}_{D,D'} \cap \mathcal{L}_{D,D'}=\emptyset$ is a sufficient condition for the non-existence of compact Clifford--Klein forms by Proposition~\ref{Prop:mainprop}. In the completely solvable case, so is the condition $\mathcal{P}^c_{D,D'}\cap\mathcal{L}_{D,D'}^c=\emptyset$ by Proposition~\ref{Prop:csmainprop}.
\end{rem}

\stitle{The spaces which correspond to Case~(II)(a) in Fact~\ref{Fact:classification}}

We show that the spaces do not admit compact Clifford--Klein forms in this case.

\begin{enumerate}
  \renewcommand{\labelenumi}{(\roman{enumi})}
 \item The case $(D,D')=(\pm\diag{1,\nu},\diag{1,-\nu})$\quad ($\nu>0$).\\
\claim{$\mathcal{L}_{D,D'}=\emptyset$.}\\
 Take any $C\in \mathcal{L}_{D,D'}$ and put $V_1:=\xspan{\R}{X_1,Y_1}\subset \fh_2$ and $V_2:=\xspan{\R}{X_2,Y_2}\subset \fh_2$. Note that the eigenvalues of $W_t$ is $\{e^{\pm \nu i t}, e^{\pm t}\}$ or $\{e^{\pm \nu t}, e^{\pm it}\}$. Since $\fl'_C$ is $W_{t_0}$-invariant for some $t_0\in\R^\times$ and $W_{t_0}|_{V_1}, W_{t_0}|_{V_2}$ do not have common eigenvalues, we have $\fl'_C=(\fl'_C\cap V_1)\oplus(\fl'_C\cap V_2)$ (see Note~\ref{Note:linnote}). On the other hand, $V_1$ or $V_2$ does not admit non-trivial $W_{t_0}$-invariant subspaces. Then we have $\mathcal{L}_{D,D'}=\emptyset$.
       \begin{note}\label{Note:linnote}
	For $\mathbb{K}=\R$ or $\C$, let $A\in M(n,\mathbb{K})$ be a matrix. Suppose $\mathbb{K}^n=V_1\oplus V_2$ is an $A$-invariant decomposition such that $A|_{V_1}$ and $A|_{V_2}$ do not have common complex eigenvalues. For an $A$-invariant subspace $V\subset \mathbb{K}^n$, we have $V=(V\cap V_1)\oplus (V\cap V_2)$.\end{note}
 \item The case $(D,D')=(\diag{1,-\nu},\diag{1,-\nu})$\quad ($\nu>0$,\ $\nu\neq 1$).\\
       In this case, $\fg_{D,D'}$ is completely solvable (Lemma~\ref{Lem:eigenvalues}(3)).\\
       \claim{$\mathcal{P}^c_{D,D'}\cap \mathcal{L}^c_{D,D'}=\emptyset$.}\\  A direct calculation leads us $\mathcal{P}^c_{D,D'}=\{\pm I_2,\pm I_{1,1}\}$. Therefore, we have $\tr D'C=\pm(1\pm \nu)\neq 0$ for any $C\in\mathcal{P}^c_{D,D'}$ and so Claim holds.
 \item The case $(D,D')=(\diag{-1,\nu},\diag{1,-\nu})$\quad ($\nu>0$,\ $\nu\neq 1$).\\
       \claim{$\mathcal{P}_{D,D'}=\emptyset$.}\\
 In this case, for $t\in\R$ we have:
 \[
 A_t=\diag{\cos t,\cos \nu t},\quad B_t=\diag{\sin t, -\sin\nu t}.
 \]
 Therefore, for any $C=\begin{pmatrix}
	 a&b\\ c&d
	\end{pmatrix}\in M(2,\R)$, a direct calculation implies:
 \begin{align*}
  \det (A_t+B_tC)=&\dfrac{1}{2}\left(1+(ad-bc)\right)\cos(t+\nu t)+\dfrac{1}{2}\left(a+d\right)\sin(t+\nu t)\\&+\dfrac{1}{2}\left(1-(ad-bc)\right)\cos (t-\nu t)+\dfrac{1}{2}\left(a-d\right)\sin(t-\nu t).
 \end{align*}
 Then Claim is a consequence of the following:
\end{enumerate}

\begin{note}
 For $A, B,b,d\in\R$, and $a,c\in\R^\times$, put $f(t):=A\sin(a t+b)+B\sin(c t+d)$. Then we have $f(t)=0$ for some $t\in \R$.
\end{note}

\stitle{The spaces which correspond to Case~(II)(b) in Fact~\ref{Fact:classification}}

We show that the spaces do not admit compact Clifford--Klein forms in this case.

\begin{lem}\label{Lem:IIlem}
 Put $(D,D')=(Q_\nu,Q_{-\nu})$ for $\nu>0$. Then we have $\mathcal{P}_{D,D'}\cap \mathcal{L}_{D,D'}=\emptyset$.
\end{lem}

\begin{proof}

First, for $t\in\R$ we have:
 \[
 W_t=
 \begin{pmatrix}
  \cosh t&&&\sinh t\\
  &\cosh t&\sinh t&\\
  &\sinh t&\cosh t&\\
  \sinh t&&&\cosh t
 \end{pmatrix}
 \begin{pmatrix}
  \cos \nu t&&-\sin \nu t &\\
  &\cos \nu t&&\sin \nu t\\
  \sin \nu t&&\cos \nu t&\\
  &-\sin \nu t&&\cos \nu t
 \end{pmatrix}.
 \]
 Set $V_\pm:=\fl'_{\pm Q_0}$, namely, $V_+=\xspan{\R}{X_1+Y_2,X_2+Y_1}$ and $V_-=\xspan{\R}{X_1-Y_2,X_2-Y_1}$. Note that $\tr W|_{V_\pm}=\pm2$ and that two dimensional subspace $V\subset \R^{2n}$ is $W$-invariant if and only if $V=V_\pm$.\par
 To prove this lemma, we show the following:\\
 {\bf Claim.} $\mathcal{L}_{D,D'}\subset \set{\begin{pmatrix}
						 a&b\\c&d
						\end{pmatrix}\in SL(2,\R)}{b+c=0}$.\\
$\because$) Let $C\in \mathcal{L}_{D,D'}$ and take $t_0\in\R^\times$ such that $\fl'_C$ is $W_{t_0}$-invariant and $\det(W_{t_0}|_{\fl'_C})=\pm1$.\\
{\bf Subclaim~1.} $C\neq \pm Q_0$.\par
 This subclaim follows from $\det(W_{t_0}|_{V_\pm})=e^{\pm 2{t_0}}$.\\
{\bf Subclaim~2.} $\nu t_0\in\pi\Z$.\par
Assume $\nu t_0\not\in\pi\Z$. Since the eigenvalues of $W$ are $\{1\pm \nu i, -1\pm \nu i\}$ by Lemma~\ref{Lem:eigenvalues}(2), the eigenvalues of $W_{t_0}$ are distinct, and so $\fl'_C$ is $W$-invariant. Since $\fl'_C$ is two-dimensional, we have $\fl'_C=V_+$ or $V_-$, which contradicts Subclaim~1 and so we have proven the Subclaim~2.\par
Since $\fl'_C$ is $W_{t_0}$-invariant, by Note~\ref{Note:linearlem} and Subclaim~2, we have:
 \begin{align*}
  \text{$\fl'_C$ is $W_{t_0}$-invariant}&\iff
  (C, -I_2) W_{t_0}
 \begin{pmatrix}
  I_2\\ C
 \end{pmatrix}=O\\
 &\iff (C, -I_2 )
  \left((\cosh t_0)I_4+(\sinh t_0)
  \begin{pmatrix}
  0&Q_0\\
  Q_0&0
  \end{pmatrix} 
\right)
  \begin{pmatrix}
  I_2\\ C
  \end{pmatrix}=O \\
 &\iff (C, -I_2 )
 \begin{pmatrix}
  0&Q_0\\
  Q_0&0
 \end{pmatrix} 
 \begin{pmatrix}
  I_2\\ C
 \end{pmatrix}=O \\
 &\iff
 CQ_0C=Q_0\\
  &\iff
 C\in\{\pm Q_0\} \sqcup \set{\begin{pmatrix}
						 a&b\\c&d
						\end{pmatrix}\in SL(2,\R)}{b+c=0}.
 \end{align*}
By Subclaim~1, we have shown Claim.\par
 Finally, we prove the lemma. Take any
 $C=\begin{pmatrix}
     a&b\\c&d
    \end{pmatrix}\in \mathcal{L}_{D,D'}$. By a direct calculation, we have:
 \begin{align*}
  \det (A_t+B_tC)&=\dfrac{(b+c)}{2}\sinh 2t+\dfrac{(a-d)}{2}\sin 2\nu t+\cos 2\nu t(\cosh^2 t-(ad-bc)\sinh^2 t)\\
  &=\cos 2\nu t+\frac{a-d}{2}\sin 2\nu t.
 \end{align*}
 Therefore, $\det(A_t+B_t C)=0$ for some $t\in \R$, so we obtain $\mathcal{P}_{D,D'}\cap \mathcal{L}_{D,D'}=\emptyset$
\end{proof}

\stitle{The spaces which correspond to Case~(II)(c) in Fact~\ref{Fact:classification}}

In this case, the spaces do not admit compact Clifford--Klein forms.
\begin{enumerate}
  \renewcommand{\labelenumi}{(\roman{enumi})}
 \item The case $(D,D')= \left(\begin{pmatrix}\ep&-1\\-1&0\end{pmatrix},\begin{pmatrix}0&-1\\ -1&\ep\end{pmatrix}\right)$, where $\ep=\pm1$.\\
In this case, $\fg_{D,D'}$ is completely solvable (Lemma~\ref{Lem:eigenvalues}(3)).

\claim{$\mathcal{P}^c_{D,D'}\cap \mathcal{L}^c_{D,D'}=\emptyset$.}\\
 By a direct calculation, we have $\mathcal{P}^c_{D,D'}=\{\pm Q_0\}$. Thus, the claim follows from $\tr D'C=\pm 2$ for $C\in \mathcal{P}_{D,D'}^c$.
 \item The case $(D,D')= \left(\begin{pmatrix}\ep&-1\\-1&0\end{pmatrix},\begin{pmatrix}0&1\\ 1&-\ep\end{pmatrix}\right)$, where $\ep=\pm1$.\\
\claim{$\mathcal{P}_{D,D'}=\emptyset$.}\\
We have: 
\[
 A_t=\begin{pmatrix}
      \cos t& 0\\
      \varepsilon t\sin t&\cos t
     \end{pmatrix},\quad
     B_t=\begin{pmatrix}
	  0&\sin t\\
	  \sin t&-\varepsilon t\cos t
	 \end{pmatrix}.
\]
 Let $C=\begin{pmatrix}
	 a&b\\c&d
	\end{pmatrix}\in M(2,\R)$. If $d=0$, we have $\det(A_t+B_t C)=(\cos t+c\sin t)(\cos t+b\sin t)$. Hence $\det (A_t+B_tC)=0$ holds for some $t\in\R$ and so $C\not\in\mathcal{P}_{D,D'}$. Then we assume $d\neq 0$. For $m\in\Z$, we have $\det(A_{2m\pi}+B_{2m\pi}C)=1-2m\pi d\varepsilon$. Then there exist $t_1, t_2\in\R$ satisfying $\det(A_{t_1}+B_{t_1}C)>0,\ \det(A_{t_2}+B_{t_2}C)<0$. Hence we have $\det(A_t+B_tC)=0$ for some $t\in\R$ by the intermediate value theorem and so $C\not\in\mathcal{P}_{D,D'}$. Therefore we have $\mathcal{P}_{D,D'}=\emptyset$.
\end{enumerate}

\stitle{The spaces which correspond to Case~(II)(b) in Fact~\ref{Fact:classification}}

In this case, the spaces admit compact Clifford--Klein forms.

\begin{enumerate}
  \renewcommand{\labelenumi}{(\roman{enumi})}
 \item The case $(D,D')=(I_{1,1},I_{1,1})$.\\
       In this case, $\fg_{D,D'}$ is completely solvable (Lemma~\ref{Lem:eigenvalues}(3)), so we are enough to show that the space satisfies the condition (b) in Proposition~\ref{Prop:csmainprop}. Set $(C, w):=(I_2,0)\in M(2,\R)\times \fh$, then the condition (b)(i) $CD'C=D$ is clear. By Note~\ref{Note:LCwform}, we have $L_{C,w}\simeq S_M$ for $M:=\diag{1,-1,0}$. By Example~\ref{ex:lattice}, $L_{C,w}$ admits a uniform lattice and so the condition (b)(ii) holds.
 \item The case $(D,D')=(-I_{1,1}, I_{1,1})$.\\
It is enough to show the conditions (b) in Proposition~\ref{Prop:mainprop}. Put $C:=Q_0$. Since we have $A_t=(\cos t) I_2$ and $B_t=(\sin t)I_{1,1}$, we get $\det (A_t+B_t C)=\sin ^2t+\cos ^2t=1$, which implies the condition (b)(i). By Note~\ref{Note:structurelcw}, we have $L_C\simeq \R^3$ and $\A_\ell=\id$, where $\ell=(2\pi,e)\in G_{D,D'}$. Then the subgroup $L_C$ has an $\A_\ell$-invariant uniform lattice $\Gamma\simeq \Z^3$. Then the condition (b)(ii) holds.
\end{enumerate}

\subsubsection{isometry group}

First, we calculate isometry group of the symmetric spaces.

\begin{prop}\label{Prop:bunruiteiri}
 For matrices $D,D'\in \sym^{(\text{reg})}(2,\R)$, assume the signature of $D'$ is $(1,1)$. Then the following list gives a complete class representatives of symmetric triple $(\fg_{D,D'},\sigma,g)$.
  \begin{enumerate}
  \item $(D,D')=(\pm\diag{1,\nu},\diag{1,-\nu})\quad (\nu>0)$,\\ $(D,D')=(\pm\diag{1,-\nu},\diag{1,-\nu})\quad (\nu>0, \ \nu\neq 1)$, 
  \item $(D,D')=\left(Q_\nu,Q_{-\nu}\right)\quad (\nu>0)$,
  \item $(D,D')=\left(\begin{pmatrix}
	    \pm1&-1\\
	    -1&0
	   \end{pmatrix},
	\begin{pmatrix}
	 0& -1\\ -1&\pm 1
	\end{pmatrix}\right),\
	\left(\begin{pmatrix}
	    \pm1&-1\\
	    -1&0
	      \end{pmatrix},
	\begin{pmatrix}
	 0&1\\ 1&\mp1
	\end{pmatrix}\right)$,
   \item $(D,D')=(\pm I_{1,1},I_{1,1})$.
  \end{enumerate}
\end{prop}

\begin{prop}\label{Prop:isomgdd}
 For a pseudo-Riemannian symmetric space $(G_{D,D'}/H,\sigma,g)$, we have:
 \[
  \Gi\simeq \begin{cases}
	     (\gen{-I_2,I_{1,1}}\times\Z/2\Z)\ltimes \Gt\quad\text{(the spaces in the list (1))}\\
	     (\gen{-I_2}\times \Z/2\Z)\ltimes \Gt \quad\text{(the spaces in the list (2) and (3))}\\
	     (O(1,1)\times \Z/2\Z)\ltimes \Gt \quad\text{(the spaces in the list (4))}\\
	    \end{cases}
 \]
\end{prop}

\begin{proof}
 By Lemma \ref{Lem:transvection}, it follows from Lemma~\ref{Lem:AHt}.
\end{proof}

In the following, we calculate the group $\Hi$.

\begin{lem}\label{Lem:autolem}
 Using the basis $(W,X_1,X_2,Y_1,Y_2,Z)$ of $\fgt\simeq \R\ltimes \fh_2$, we have:
 \[
 \Aut(\fg_{D,D'},\sigma,g)=\set{ \begin{pmatrix}
  k&&&\\
  v&P&&\\
&&{kP^T}\inv&\\
  s&w&&k
       \end{pmatrix}\in GL(6,\R)}{\begin{gathered}
				   P\in O(1,1)\\ k\in\{\pm1 \},\ v\in\R^2\\
				   PDP^T=D
				  \end{gathered}},
 \]
 where $s:=(k/2)v^TI_{1,1}v,\ w:=kv^TI_{1,1}P$.
\end{lem}

\begin{proof}
 Let $\phi\in \Aut(\fgt,\sigma,g)$. We use $(W,X_1,X_2,Y_1,Y_2,Z)$ as the basis of $\fgt$. Then the representation matrix of $\phi$ is:
 \[
\phi=
 \begin{pmatrix}
  \ell_1&&&\\
  v&P&&\\
  &&Q&\\
  s&w&&\ell_2
 \end{pmatrix}\in GL(6,\R),
 \]
 where $v,w^T\in\R^2,\ s,\ell_1,\ell_2\in\R,\ P,Q\in GL(2,\R)$. Then we have:
 \begin{align}
  &\text{$g$ is $\phi$ invariant} \iff \phi^Tg\phi=g
  &&\iff  \begin{cases}P^T I_{1,1} P=I_{1,1},\\
	     \ell_1\ell_2=1,\\
	     v^TI_{1,1}P=\ell_1w,\\
	     v^TI_{1,1}v=2\ell_1s,\\
	     \end{cases}\\
  &[\phi(X_i),\phi(Y_j)]=\phi([X_i,Y_j])\ (i,j=1,2)&&\iff P^TQ=\ell_2I_2,\label{3}\\
  &[\phi(W),\phi(Y_i)]=\phi([W,Y_i])\ (i=1,2)&&\iff I_{1,1}P=\ell_1 QI_{1,1},\ v^TQ=w I_{1,1}, \label{4}\\
  &[\phi(W),\phi(X_i)]=\phi([W,X_i])\ (i=1,2)&&\iff \ell_1PD=DQ, \label{5}
 \end{align}
By the first identity of (0.1), we have $P\in O(1,1)$. On the other hand $Q=\ell_2{P^T}\inv$ by (\ref{3}). By the determinant relation of (0.3) and (0.4), we have $\ell_1=\pm1$. By the second identity of (0.1), we have $\ell_1=\ell_2=\pm 1$. We put $k:=\ell_1$, then $s:=(k/2)v^TI_{1,1}v$ and $w=kv^TI_{1,1}P$ follows from (0.1). Moreover, a direct calculation leads us $\phi\in\Aut(\fgt,\sigma,g)$ if $\phi$ satisfies the conditions of Lemma~\ref{Lem:autolem}.
 \end{proof}

\begin{lem}\label{Lem:Htisom}
 By the identification $\Hi\simeq \Aut(\fgt,\sigma,g)$, we have:
 \[
  \Ht=\set{   \begin{pmatrix}
  1&&&\\
  -I_{1,1}y&I_2&&\\
	       &&I_2&\\
  (y^TI_{1,1}y)/2&-y^T&&1
 \end{pmatrix}
}{y\in\R^n}
 \]
\end{lem}

\begin{proof}
 Note that $\Ht=\set{\exp (\sum_{i=1}^ny_iY_i)}{y\in\R^n}$. The adjoint action of $\exp(\sum_{i=1}^ny_iY_i)$ is:
 \[
  \ad(Y_i)=\begin{pmatrix}
  0&&&\\
  -I_{1,1}e_i&0&&\\
	       &&0&\\
  &-e_i^T&&0
 \end{pmatrix}
 \]
 This lemma follows from Lemma~\ref{Lem:hishape}.
\end{proof}

\begin{lem}\label{Lem:AHt}
 We put $A:=\set{\diag{k,P,{kP^T}\inv,k}}{P\in O(1,1), k\in \{\pm1\}, PDP^T=D}\subset \Hi$. Then $A$ is a closed subgroup of $\Hi$ and we have $\Hi= A\ltimes \Ht$.
\end{lem}

\begin{proof}
By a direct calculation, we see $A$ is a closed subgroup of $\Hi$. Note that $\Ht$ is a closed normal subgroup of $\Hi$. By Lemma~\ref{Lem:Htisom}, we have $\Ht\cap A=\{e\},\ \Hi=A\Ht$. Then this lemma follows from Note~\ref{Note:ANlem}.
\end{proof}

\begin{lem}
We have:
 \begin{align*}
  \set{\diag{k,P,{kP^T}\inv,k}}{P\in O(1,1), PDP^T=D, k\in \{\pm1\}}\\
  \simeq \begin{cases}\Z/2\Z\times O(1,1)\quad\text{(the spaces in the list (1))}\\ \Z/2\Z\times \gen{-I_2}\quad\text{(the spaces in the list (2) and (3))}\\\Z/2\Z\times \gen{-I_2,I_{1,1}}\quad\text{(the spaces in the list (4))}\end{cases}
 \end{align*}
\end{lem}

\begin{proof}
In the case $D=\pm I_{1,1}$, it is clear. In the case $D\neq I_{1,1}$, it follows from $PDP^T=D\siff P=\pm I_2$.
\end{proof}

\begin{cor}
 The space $\Gt/H$ admits compact Clifford--Klein forms if and only if so does $\Gi/H$.
\end{cor}

\begin{proof}
 In the case $D=I_{1,1}$, the space $\Gt/\Ht$ admits compact Clifford--Klein forms. In the case $D\neq \pm I_{1,1}$, we have $\Lie{\Gi}=\Lie{\Gt}$ by Proposition~\ref{Prop:isomgdd}, so it does not admit compact Clifford--Klein forms by Note~\ref{Note:component} and .
\end{proof}

\section{Kobayashi's conjecture about standard quotients}\label{Section:Onkobayashisconjecture}

There have been attempts to extend Kobayashi's theory on discontinuous groups for reductive cases [17-23] to non-reductive cases such as Baklouti-K\'edim\cite{BK}, Kath-Olbrich\cite{OKCK}, Kobayashi-Nasrin\cite{KN}, Lipsman\cite{lipsman}, Nasrin\cite{nasrin}, Yoshino\cite{3step} and so on.
In this section, we examine a `solvable analogue' of Kobayashi's conjecture (Conjecture \ref{Conj:Kobayashiconjecture}) and see an evidence that the assumption `reductive type' in Kobayashi's conjecture is crucial. 

\begin{ex}{\label{ex:ceforkobconj}}
 We put $n=3$ and $(D,D'):=(\diag{-1,-1,2}, \diag{1,1,-2})$. Then $G_{D,D'}/H$ admits compact Clifford--Klein forms and does not admit constructors.
\end{ex}

\begin{proof}
 First, we check $G_{D,D'}/H$ does not admit constructors. Assume $G_{D,D'}/H$ admits constructors. Then there exists $C\in M(3,\R)$ such that $CD'C=D$ by Proposition~\ref{Prop:structurethm}. However, we obtain $(\det C)^2=-1$, which contradicts $C\in M(3,\R)$.\par
 Next, we check $G_{D,D'}/H$ admits compact Clifford--Klein forms by using Proposition~\ref{Prop:mainprop}. We set:
\[
 C\coloneqq \begin{pmatrix}0&1&0\\1&0&-2\\0&-1&0\end{pmatrix}.
\]
 It is enough to check that the conditions (b)(i) and (ii) in Proposition~\ref{Prop:mainprop}. A direct calculation leads us that $A_t=\diag{\cos t,\cos t,\cos 2t}$, $B_t=\diag{\sin t,\sin t, -\sin 2t}$, $\det (A_t+B_tC)=\cos^2(2t)+\sin^2 (2t)=1$ and so the condition (i) holds. Set $t_0\coloneqq 2\pi\in\R$ and $\ell:=(t_0,e)\in G_{D,D'}$. Then we have $W_{t_0}=I_6$ and so $\A_{\ell}=\id$. Then $L_C\simeq H_1\times\R$ (Note~\ref{Note:structurelcw}) has an $\A_\ell$-invariant uniform lattice $\Gamma\simeq H_1(\Z)\times \Z$, and so the condition (ii) is satisfied.
\end{proof}

\section*{Acknowledgments}

The author would like to thank his supervisors Professor Toshiyuki Kobayashi and Professor Taro Yoshino for many constructive comments. This work was supported  by the Program for Leading Graduate Schools, MEXT, Japan.

\bibliography{thesis_for_arxiv}

\end{document}